\documentclass[12pt]{amsart}

\usepackage{amsthm,amssymb,enumitem,mathtools,color}
\usepackage[dvipsnames]{xcolor}
\usepackage{appendix}

\usepackage{tikz}

\usepackage[margin=1in,letterpaper,portrait]{geometry}

\theoremstyle{definition}
\newtheorem{thm}{Theorem}[section]
\newtheorem{prop}[thm]{Proposition}
\newtheorem{example}[thm]{Example}

\newtheorem{lemma}[thm]{Lemma}
\newtheorem{defn}[thm]{Definition}
\newtheorem{cor}[thm]{Corollary}

\newtheorem{question}[thm]{Question}
\newtheorem{rem}[thm]{Remark}

\DeclareMathOperator{\Des}{Des}
\DeclareMathOperator{\Dtop}{Dtop}

\DeclareMathOperator{\Pin}{Pin}

\DeclareMathOperator{\Pk}{Pk}
\DeclareMathOperator{\std}{std}

\newcommand{\pfrk}{\mathfrak{p}}
\newcommand{\Pcal}{\mathcal{P}}
\newcommand{\ww}{\widehat{w}}

\renewcommand\labelenumi{\theenumi.}

\title{The Pinnacle Set of a Permutation}

\author[Davis]{Robert Davis}%
\address{Department of Mathematics, Michigan State University, East Lansing, MI 48824}
\email{davisr@math.msu.edu}

\author[Nelson]{Sarah A.~Nelson}%
\address{School of Mathematics and Computing Sciences, Lenoir-Rhyne University, Hickory, NC 28601}
\email{sarah.nelson@lr.edu}

\author[Petersen]{T. Kyle Petersen$^*$}%
\address{Department of Mathematical Sciences, DePaul University, Chicago, IL 60614}
\email{tpeter21@depaul.edu}%

\author[Tenner]{Bridget E.~Tenner$^*$}%
\address{Department of Mathematical Sciences, DePaul University, Chicago, IL 60614}
\email{bridget@math.depaul.edu}%

\thanks{$^*$ Partially supported by Simons Foundation Collaboration Grants for Mathematicians.}

\begin{document}

\begin{abstract}
The peak set of a permutation records the indices of its peaks. These sets have been studied in a variety of contexts, including recent work by Billey, Burdzy, and Sagan, which enumerated permutations with prescribed peak sets. In this article, we look at a natural analogue of the peak set of a permutation, instead recording the values of the peaks. We define the ``pinnacle set'' of a permutation $w$ to be the set $\{w(i) : i \text{ is a peak of }w\}$. Although peak sets and pinnacle sets mark the same phenomenon for a given permutation, the behaviors of these sets differ in notable ways as distributions over the symmetric group.
In the work below, we characterize admissible pinnacle sets and study various enumerative questions related to these objects.
\end{abstract}

\maketitle

\section{Introduction}\label{sec:intro}

Let $S_n$ denote the set of permutations of $[n] = \{1,2,\ldots, n\}$, which we will always write as words, $w = w(1)w(2)\cdots w(n)$. An \emph{ascent} of a permutation $w$ is an index $i$ such that $w(i)< w(i+1)$, while a \emph{descent} is an index $i$ such that $w(i) > w(i+1)$. A \emph{peak} is a descent that is preceded by an ascent, whereas a \emph{valley} is an ascent that is preceded by a descent. This terminology refers to the shape of the graph of $w$, that is, the set of points $(i,w(i))$. 


\begin{example}
The descents of $315264 \in S_6$ are $1$, $3$, and $5$, and the ascents are $2$ and $4$. The peaks are $3$ and $5$, while the valleys are $2$ and $4$.
\end{example}

The \emph{descent set} of a permutation $w$, denoted $\Des(w)$, is the collection of its descents,
\[
 \Des(w) = \{ i : w(i) > w(i+1) \} \subseteq [n-1],
\]
while the \emph{peak set} of a permutation $w$, denoted $\Pk(w)$, is the collection of its peaks,
\[
 \Pk(w) = \{ i : w(i-1) < w(i) > w(i+1) \} \subseteq \{2,3,\dots,n-1\}.
\]
Note in particular that the descent set completely determines the peak set:
\[
 \Pk(w) = \{ i > 1 : i \in \Des(w) \text{ and } i-1 \notin \Des(w)\}.
\]
Any subset of $\{1,2,\ldots,n-1\}$ is the descent set of some permutation in $S_n$, but the same cannot be said for peak sets. First of all, peaks cannot occur in the first or last positions of a permutation, so $\Pk(w) \subseteq \{2,\ldots,n-1\}$ for any $w \in S_n$. Moreover, peaks cannot occur in consecutive positions, so if $i \in \Pk(w)$ then $i\pm1 \not\in \Pk(w)$. This characterization of peak sets, as subsets of $\{2,\ldots,n-1\}$ with no consecutive elements, turns out to imply that the number of distinct peaks sets is given by the Fibonacci numbers.

It has long been known that counting permutations according to the number of descents gives rise to the Eulerian numbers, while the number of permutations with a given descent set is also well known; see, e.g., \cite[Example 2.2.4]{StanleyEC1}.
More recently Billey, Burdzy, and Sagan~\cite{BilleyBurdzySagan} considered the related enumerative question for peaks: how many permutations in $S_n$ have a given peak set? One of their results is that for a fixed set $S$, the number of $w \in S_n$ for which $\Pk(w) = S$ is a power of two times a polynomial in $n$, and they give techniques for explicit computation of this polynomial in special cases.
As a follow up to this work, Kasraoui~\cite{Kasraoui} verified their related conjecture about which peak sets of a given cardinality maximize the number of permutations in $S_n$ for a given $n$.

In the present article, we study analogous questions related to peaks, but rather than tracking peaks by their positions ($x$-coordinates in the graph of the permutation), we use their values ($y$-coordinates).

\begin{defn}
A \emph{pinnacle} of a permutation $w$ is a value $w(i)$ such that $w(i-1)< w(i) > w(i+1)$; equivalently, $j$ is a pinnacle of $w$ if and only if $w^{-1}(j) \in \Pk(w)$. The \emph{pinnacle set} of $w$ is
\[
 \Pin(w) = \{ w(i) : i \in \Pk(w) \}.
\]
\end{defn}

Certainly $|\Pk(w)|=|\Pin(w)|$, but the sets themselves need not be the same, as we now demonstrate. 

\begin{example}\label{ex:peaks and pinnacles}
	If $w = 315264$, then $\Pk(w) = \{3,5\}$ and $\Pin(w) = \{5,6\}$.
\end{example}

The definition of pinnacle sets leads naturally to questions about the value
\begin{equation}\label{eqn:counting permutations with a given pinnacle set}
 p_S(n) := | \{ w \in S_n : \Pin(w) = S\} |.
\end{equation}
While similar notation was used to denote the peak polynomial, e.g. in \cite{BilleyBurdzySagan, BilleyFahrbachTalmage, DiazHarrisInskoOmar}, note that $p_S(n)$ is counting the number of permutations  with a given pinnacle set $S$ in this paper.
The questions we address in this article are the following.
\begin{question}
When is $p_S(n) > 0$? That is, which sets $S$ are the pinnacle set of some permutation in $S_n$?\label{q:0}
\end{question}
\begin{question}
Given a pinnacle set $S \subseteq [n]$, how do we compute $p_S(n)$? \label{q:1}
\end{question}
\begin{question}
For a given $n$, what choice of $S \subseteq [n]$ maximizes or minimizes $p_S(n)$? \label{q:2} 
\end{question}

In Section~\ref{sec:admissible} we identify conditions under which a set $S$ is the pinnacle set for some permutation, fully answering Question \ref{q:0}. 

\begin{defn}\label{def:admissible}
A set $S$ is an \emph{$n$-admissible pinnacle set} if there exists a permutation $w \in S_n$ such that $\Pin(w) = S$.
If $S$ is $n$-admissible for some $n$, then we simply say that $S$ is \emph{admissible}.
\end{defn}


The empty set is always an $n$-admissible pinnacle set, because it is the pinnacle set of the identity permutation. Examples of nonempty admissible pinnacle sets are shown in Table \ref{tab:pinn}. The main result about admissible pinnacle sets is the following.

\begin{thm}[Admissible pinnacle sets]\label{thm:pinsets}
Let $S$ be a nonempty set of integers with $\max S = m$. Then $S$ is an admissible pinnacle set if and only if both
\begin{enumerate}
\item $S\setminus\{m\}$ is an admissible pinnacle set, and
\item $m > 2|S|$.
\end{enumerate}
Moreover, there are $\binom{m-2}{\lfloor m/2 \rfloor}$ admissible pinnacle sets with maximum $m$, and 
\[
 1+\sum_{m=3}^n \binom{m-2}{\lfloor m/2 \rfloor} = \binom{n-1}{\lfloor (n-1)/2 \rfloor},
\]
admissible pinnacle sets $S \subseteq [n]$.
\end{thm}

Our characterization of admissible pinnacle sets is in contrast to the characterization of peak sets mentioned earlier. Whereas the number of peak sets is given by the Fibonacci numbers, here we get a central binomial coefficient. 

%
%
%

In Section \ref{sec:rec} we develop both a quadratic and a linear recurrence for $p_S(n)$, which partially answers Question~\ref{q:1}. Further, we identify the following bounds for $p_S(n)$ partially answering Question~\ref{q:2}; the sets which achieve the tight bounds are constructed in Section~\ref{subsec: some formulas}.

\begin{thm}[Bounds on $p_S(n)$]\label{thm:bounds}
Let $d$ and $n$ be any positive integers such that $2d < n$. Then for any admissible pinnacle set $S\subseteq[n]$ such that $|S|=d$, we have the tight upper bound
\begin{equation}\label{eq:upbound}
p_S(n) \leq d!\cdot (d+1)! \cdot 2^{n-2d-1} \cdot S(n-d,d+1),
\end{equation}
where $S(\cdot,\cdot)$ denotes the Stirling number of the second kind,
and the tight lower bound
\begin{equation}
p_S(n) \geq 2^{n-d-1}.
\end{equation}
\end{thm}

It follows that across all admissible pinnacle sets $S\subseteq [n]$, the cardinality $\#\{w \in S_n : \Pin(w) = S\}$ has a uniform lower bound of $2^{\lfloor n/2 \rfloor}$, while the upper bound is achieved for the particular value of $d = |S|$ that maximizes Equation~\eqref{eq:upbound}. While this choice of $d$ appears to be close to $n/3$, we have no simple expression for $d$ in terms of $n$. Our best approximation so far (having checked as high as $n=5000$) is $d \approx n/3.13$.
Section \ref{sec:conclude} contains more discussion on this question, as well as other open questions.



We close the introduction with three remarks.

\begin{rem}[Descent topsets]\label{rem:topsets}
Just as the pinnacle set records the values that sit at peaks, the \emph{descent topset} records the values that sit at descents: 
\begin{align*}
\Dtop(w) &= \{ w(i) : w(i) > w(i+1) \} \\
&= \{ w(i) : i \in \Des(w) \}.
\end{align*}
Descent topsets and related ideas have appeared sporadically in the literature on permutation statistics, e.g., see \cite{EhrenborgSteingrimsson, EhrenborgSteingrimssonAlternating, FoataZ, HNTT, KitaevMansourRemmel, NovelliThibonWilliams, SteinWilliams}. Enumeration of permutations with a fixed topset is considered in work of Ehrenborg and Steingr\'imsson \cite{EhrenborgSteingrimsson}, via a correspondence with excedance sets. The question of enumeration by pinnacle sets does not appear to have been addressed in the literature.

While the peak set $\Pk(w)$ is completely determined by the descent set $\Des(w)$, the pinnacle set is not determined by the descent topset.  For example, suppose $w=3175264$ and $v = 7651324$. Then we have $\Dtop(w) = \{3,5,6,7\} = \Dtop(v)$, yet $\Pin(w) =\{6,7\}$ while $\Pin(v) = \{3\}$. Thus it seems unlikely that enumeration results for pinnacle sets will follow directly from results for descent topsets.

Since there are no enumerative results connecting pinnacle sets to descent topsets, we will briefly describe the enumeration of permutations by descent topsets due to Ehrenborg and Steingr\'imsson in Appendix \ref{sec:topsets}. \end{rem}

\begin{rem}[Descent algebras and peak algebras]
Grouping permutations according to descent sets or peak sets leads to interesting and well-studied algebraic structures. For example, the group algebra of the symmetric group has a subalgebra known as \emph{Solomon's descent algebra} \cite{Solomon}, with linear basis given by sums of descent classes, i.e., by the elements 
\[
 y_I = \sum_{\substack{w \in S_n \\ \Des(w) = I}} w.
\]
A subalgebra of Solomon's descent algebra known as the \emph{peak algebra} has a basis whose elements are sums of peak classes, i.e.,
\[
 z_J = \sum_{\substack{w \in S_n \\ \Pk(w) = J}} w.
\]
There are a number of papers investigating the connections between descent algebras and peak algebras, e.g., \cite{AguiarNymanOrellana, BilleraHsiaoVW, GarsiaReutenauer, Nyman, Schocker}. It is natural to wonder whether some similar algebraic structures can be associated to descent topsets or pinnacle sets. However, taking sums of descent topset classes or sums of pinnacle classes does not yield a subalgebra of the group algebra in general.
\end{rem}

\begin{rem}[Descent and Peak Polynomials]
Two polynomials that record information related to pinnacles are the peak polynomial and the descent polynomial.
A detailed analysis of the roots and expansions of peak polynomials in terms of binomials in given in \cite{BilleyFahrbachTalmage}.
This work proves that the coefficients of the peak polynomials, expanded in the binomial coefficient basis centered at the maximum of the corresponding subset, are all nonnegative.
Additionally, a recursive formula for the descent polynomial $d(S;n)$, as well as an exploration of their coefficients and roots, is given in \cite{DiazHarrisInskoOmarSagan}.
\end{rem}

\noindent\textbf{Acknowledgments.} 
The authors would like to thank Bruce Sagan for helpful conversations related to the project. In particular, he suggested the simple recursive proof for the total number of admissible pinnacle sets in Remark \ref{rem:alternate} and the recursive proof for the number of permutations with a given descent topset (proof of Theorem \ref{thm:eJ}).

\section{Admissible pinnacle sets}\label{sec:admissible}

Not every set is the peak set of a permutation. Likewise, not every set is the pinnacle set of a permutation.
For one thing, each peak must have a non-peak on each side of it, so the number of peaks must be strictly less than half the number of letters in the permutation.

\begin{lemma}[Limited number of peaks]\label{lem:halfn}
A permutation $w \in S_n$ has at most $\lfloor (n-1)/2\rfloor$ peaks. That is, $n > 2|\Pk(w)| = 2|\Pin(w)|$. 
\end{lemma}

Our goal in this section is to push this result a bit further and to completely characterize pinnacle sets.

\subsection{Characterization of admissible pinnacle sets}

Recall from Definition~\ref{def:admissible} that a set $S$ is an $n$-admissible pinnacle set if there exists a permutation $w \in S_n$ such that $\Pin(w) = S$.

\begin{example}\label{ex:nonpin}\
\begin{enumerate}\renewcommand{\labelenumi}{(\alph{enumi})}
\item The set $S = \{3,7,8\}$ is an $8$-admissible pinnacle set because $\Pin(13247586) = S$. The set $S$ is certainly not $n$-admissible for any $n<8$, because $8 \in S$.
\item For the set $S = \{3, 5, 6\}$ to be an admissible pinnacle set, there would have to be a permutation
\[ 
w = \cdots a \ x \ b_1 \cdots b_2 \ y \ c_1 \cdots c_2 \ z \ d \cdots 
\]
such that $S = \{x,y,z\}$, with $a < x > b_1$, $b_2 < y > c_1$, and $c_2 < z > d$. It is possible that $b_1 = b_2$ or $c_1 = c_2$ or both, but $a$, $b_1$, $c_1$, and $d$ must be distinct. In fact, these four values must all be less than $6$, and none can be an element of $S$. However, there are only three positive integers less than $6$ and not in $S$, so there can be no such permutation $w$. Thus $S$ is not  an admissible pinnacle set.
\end{enumerate}
\end{example}

Pinnacle sets are stable in the sense that if $S$ is an $n$-admissible pinnacle set, then $S$ is also $(n+1)$-admissible. Indeed, if $\Pin(w) = S$ for $w \in S_n$, then we can form a permutation in $S_{n+1}$ with pinnacle set $S$ by putting $n+1$ at the far left or far right of the permutation. That is, if $u = (n+1)w(1)\cdots w(n)$ and $v = w(1)\cdots w(n)(n+1)$, then
\[
 \Pin(u) = \Pin(v) = \Pin(w).
\] 
Moreover, any other way to insert $n+1$ into $w$ will give a different pinnacle set, since $n+1$ would sit at a peak. Thus a kind of converse to this stability observation is the observation that if $\max S = m$, and $S$ is an $n$-admissible pinnacle set for some $n\geq m$, then $S$ is $m$-admissible.

Extending this idea leads to the following recursive characterization of admissible pinnacle sets, which establishes the first half of Theorem \ref{thm:pinsets} from the introduction.

\begin{prop}[Admissible pinnacle sets]\label{prop:admissible}
Let set $S$ be a nonempty set of integers with $\max S = m$. Then $S$ is an admissible pinnacle set if and only if both
\begin{enumerate}
\item $S\setminus\{m\}$ is an admissible pinnacle set, and
\item $m > 2|S|$.
\end{enumerate}
Moreover, $S$ is $n$-admissible for all $n\geq m$.
\end{prop}

Some admissible pinnacle sets are shown in Table \ref{tab:pinn}.

\begin{table}[htbp]
\[
\begin{array}{c|| l |l| l| l}
m & d = 1 & d = 2 & d = 3 & d = 4 \\
\hline \hline 
\raisebox{.1in}[.2in][.1in]{}3 & \{3\} & &\\
\hline
\raisebox{.1in}[.2in][.1in]{}4 & \{4\} & &\\
\hline
\raisebox{.1in}[.2in][.1in]{}5 & \{5\} & \{3,5\}, \{4,5\} &\\
\hline
\raisebox{.1in}[.2in][.1in]{}6 & \{6\} & \{3,6\}, \{4,6\}, &\\
\raisebox{.1in}[.1in][.1in]{}& &  \{5,6\} & \\
\hline
\raisebox{.1in}[.2in][.1in]{}7 & \{7\} & \{3,7\}, \{4,7\},  & \{3,5,7\}, \{3,6,7\}, \{4,5,7\}, \\
\raisebox{.1in}[.1in][.1in]{}& & \{5,7\}, \{6,7\}  & \{4,6,7\}, \{5,6,7\} \\
\hline
\raisebox{.1in}[.2in][.1in]{}8 & \{8\} & \{3,8\}, \{4,8\},  & \{3,5,8\}, \{3,6,8\}, \{3,7,8\},\\
\raisebox{.1in}[.1in][.1in]{} & & \{5,8\}, \{6,8\},  & \{4,5,8\}, \{4,6,8\}, \{4,7,8\}, \\
\raisebox{.1in}[.1in][.1in]{} & & \{7,8\} & \{5,6,8\}, \{5,7,8\}, \{6,7,8\} \\
\hline
\raisebox{.1in}[.2in][.1in]{}9 & \{9\} & \{3,9\}, \{4,9\}, & \{3,5,9\}, \{3,6,9\}, \{3,7,9\},& \{3,5,7,9\}, \{3,5,8,9\}, \{3,6,7,9\}, \\
\raisebox{.1in}[.1in][.1in]{} & & \{5,9\}, \{6,9\},&  \{3,8,9\}, \{4,5,9\}, \{4,6,9\},  & \{3,6,8,9\}, \{3,7,8,9\},\{4,5,7,9\}, \\
\raisebox{.1in}[.1in][.1in]{} & &  \{7,9\}, \{8,9\} & \{4,7,9\}, \{4,8,9\}, \{5,6,9\}, & \{4,5,8,9\}, \{4,6,7,9\}, \{4,6,8,9\}, \\
\raisebox{.1in}[.1in][.1in]{} & & &  \{5,7,9\}, \{5,8,9\}, \{6,7,9\}, & \{4,7,8,9\},\{5,6,7,9\}, \{5,6,8,9\}, \\
\raisebox{.1in}[.1in][.1in]{} & & & \{6,8,9\}, \{7,8,9\} & \{5,7,8,9\}, \{6,7,8,9\}
\end{array}
\]
\caption{Nonempty admissible pinnacle sets $S$ with maximum element $m$ and $|S|=d$.}\label{tab:pinn}
\end{table}

In order to prove this proposition, it will be helpful to have a canonical way to construct a permutation in $S_n$ for any $n\geq m$ with a given (admissible) pinnacle set, which we describe now.
First we order the elements of $S = \{ s_1 < s_2 < \cdots < s_d\}$. Then we use these as the values of the even positions of a permutation $w$, so that $w(2i) = s_i$ for $i \in [d]$. We place the elements not in $S$ into the odd positions of $w$, in increasing order. Let $w_S$ denote the permutation we have thus formed. 

More precisely, suppose $S= \{ s_1 < s_2 < \cdots < s_d\}$ and $s_d = m$.
Define the complementary set $[m]\setminus S = \{ t_1 < t_2 < \cdots < t_{m-d}\}$. Then, for $1 \leq i \leq m$, set
\begin{equation}\label{eqn:canonical perm with given pinnacle set}
	w_S(i) := \begin{cases}
		s_j & \text{ if } i = 2j \text{ and } i \leq 2d \\
		t_j & \text { if } i = 2j-1 \text{ and } i \leq 2d \\
		t_{i-d} & \text{ if } i > 2d.
		\end{cases}
\end{equation}

Visually, we can imagine labels on a ``mountain range" diagram, an illustration of which is shown in Figure \ref{fig:mountain}.

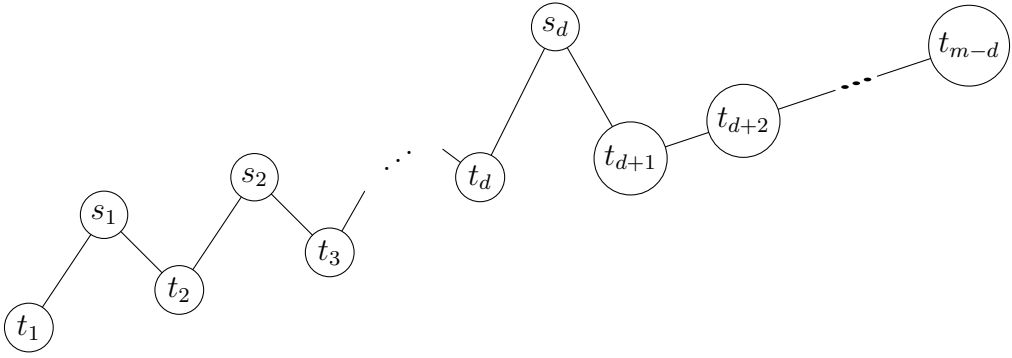
\begin{figure}[htbp]
\[
 \begin{tikzpicture}[yscale=.5]
  \draw (1,1) node[circle, draw=black, fill=white, inner sep=2] {$t_1$} 
  -- (2,4) node[circle, draw=black, fill=white, inner sep=2] {$s_1$} 
  -- (3,2) node[circle, draw=black, fill=white, inner sep=2] {$t_2$} 
  -- (4,5) node[circle, draw=black, fill=white, inner sep=2] {$s_2$}
  -- (5,3) node[circle, draw=black, fill=white, inner sep=2] {$t_3$} -- (5.5,4.75);
  \draw (5.9,5.5) node[circle, draw=white, fill=white, inner sep=4] {\reflectbox{$^{\ddots}$}};
  \draw (6.5,5.75)
  -- (7,5) node[circle, draw=black, fill=white, inner sep=2] {$t_d$}
  -- (8,9) node[circle, draw=black, fill=white, inner sep=2] {$s_d$}
  -- (9,5.5) node[circle, draw=black, fill=white, inner sep=2] {$t_{d+1}$}
  -- (10.5,6.5) node[circle, draw=black, fill=white, inner sep=2] {$t_{d+2}$}
  -- (12,7.5) node[circle, draw=white, fill=white, inner sep=2] {\ \ \ \  }
  -- (13.5,8.5) node[circle, draw=black, fill=white, inner sep=2] {$t_{m-d}$};  
  \foreach \x in {(11.85,7.4),(12,7.5),(12.15,7.6)} {\fill \x circle (.05);}
 \end{tikzpicture}
\]
\caption{Canonical construction of $w_S$, a permutation having pinnacle set $ S=\{s_1 < s_2<\cdots < s_d\}$.}\label{fig:mountain}
\end{figure}

\begin{example}
The set $S=\{5,8,9\}$ is an admissible pinnacle set. To produce $w_S$, we first set $w(2)= 5$, $w(4)= 8$, and $w(6)= 9$. Next, we position the values $\{1,2,3,4,6,7\}$ in increasing order, yielding $w = 152839467$.
\end{example}

Let us now clearly state and prove our assertion about $w_S$.

\begin{prop}[Canonical permutation with a given pinnacle set]
Let $S$ be an admissible pinnacle set with maximum $m$, and let $w_S \in S_m$ be as defined in Equation~\eqref{eqn:canonical perm with given pinnacle set}. Then $\Pin(w_S) = S$.
\end{prop}

\begin{proof}
	Suppose $S$ is an admissible pinnacle set and $w_S$ is the permutation constructed above.
	Since $S$ is admissible, for each $i \le |S|$, there are at least $i+1$ elements of $[m] \setminus S$ that are less than $s_i$.
	So, the elements $t_1,\ldots,t_{i+1}$ will always be less than $s_i$.
	This implies that when $i$ is even and $i \leq 2d$, we have $w_{i-1}w_iw_{i+1} = t_js_jt_{j+1}$ for some $j$.
	Thus, $s_j \in \Pin(w_S)$ for each $j$, and moreover, $t_j, t_{j+1} \notin \Pin(w_S)$ since there cannot be two adjacent peaks. 
	Finally, observe that $w(2d+1) w(2d+2) \cdots w(m) = t_{d+1}t_{d+2}\cdots t_{m-d}$ is an increasing sequence, so none of $t_{d+1},t_{d+2},\ldots,t_{m-d}$ will appear in $\Pin(w_S)$.
\end{proof}

\begin{proof}[Proof of Proposition~\ref{prop:admissible}]
We proceed by induction on $d=|S|$.

First recall that $\emptyset$ is an admissible pinnacle set, since it is the pinnacle set for $12\dots n$.
Next, suppose that $|S| = 1$, meaning that $S = \{m\}$.
If $S$ is an admissible pinnacle set, then $S \setminus \{m\} = \emptyset$ is an admissible pinnacle set. In addition, recall from Example~\ref{ex:nonpin} that at least two smaller numbers must be adjacent to each pinnacle. As a result, 1 and 2 can never be pinnacles. So $m \geq 3 > 2|S|$.
For the converse, consider the permutation $\pi=12\cdots(m-2)(m)(m-1) \in S_m$ where $m\geq 3$. Notice that $\Pin(12\cdots(m-2)(m)(m-1) )=\{m\}$. Thus, the converse implication also holds.
Now, assume that for some $d \ge 1$, the result holds for any set of size $d$, and consider a set $S=\{s_1 < s_2 < \cdots < s_d < s_{d+1}\} \subseteq\{1,2,3,\ldots\}$ with maximal element $s_{d+1}=m$.
Set $S' := S \setminus \{m\}$.

Suppose, first, that $S$ is an admissible pinnacle set. Let $w_S$ be the canonical permutation described by Equation~\eqref{eqn:canonical perm with given pinnacle set}, for which $\Pin(w_S) = S$. Since $w_S \in S_m$, Lemma~\ref{lem:halfn} tells us $m > 2|S|=2(d+1)$. Moreover, if we remove $m=s_{d+1} = w(2(d+1))$  from $w_S$, then we are left with a permutation $w'$ with pinnacle set $S \setminus \{m\} = S'$. Thus $S'$ is an admissible pinnacle set.

Now suppose that $S'$ is an admissible pinnacle set and that $m > 2(d+1)$. We must show that $S$ is an admissible pinnacle set. The set $S'$ has size $d$, and maximal element $s_d < m=s_{d+1}$. As $S'$ is admissible, there is a permutation $w' \in S_{m-1}$ that has pinnacle set $S'$. Let $T = [m-1]\setminus S'$, the set of non-pinnacles in $w'$. Since we are assuming $m>2(d+1)$, we have $|T| = m-1-d > d+1$. There are only $d$ peaks in $w'$, hence, by the pigeonhole principle, at least two elements of $T$ appear consecutively in $w'$. Let $w \in S_m$ be the permutation obtained by inserting $m$ between these two consecutive elements of $T$. This yields a permutation $w$ for which $\Pin(w) = S' \cup \{m\} = S$. Hence $S$ is an admissible pinnacle set.
\end{proof}

\subsection{Enumeration of admissible pinnacle sets}

We now use our characterization of admissible pinnacle sets from Proposition~\ref{prop:admissible} to count these sets.

\begin{defn}
Given nonnegative integers $m$ and $d$, define
\[\pfrk(m;d)\]
to be the number of admissible pinnacle sets with maximum element $m$ and cardinality $d$, using the convention $\pfrk(0,0) = 1$.
\end{defn}

In Table \ref{tab:admissible} we provide the numbers $\pfrk(m;d)$ for small values of $m$ and $d$.
From our characterization of admissible pinnacle sets in Proposition~\ref{prop:admissible}, we have the following recurrence for the array:
\[
\pfrk(m;d) = \begin{cases}
\hspace{.35in} 1 & \text{if } m = d = 0,\\
\sum\limits_{k<m} \pfrk(k;d-1) & \text{if } m > 2d, \text{ and}\\
\hspace{.35in} 0 & \text{otherwise.}
\end{cases}
\]

\begin{table}[htbp]
\[
\begin{array}{c || c c c c c c | l}
 m & d = 0 & d=1 & d=2 & d=3 & d=4 & d = 5  & \mbox{row sums}\\
 \hline
 \hline
\raisebox{.1in}[.2in][.1in]{} - & 1 & & & & & & 1\\
\raisebox{.1in}[.2in][.1in]{} 1 & 0 & & & & & & 0 \\
\raisebox{.1in}[.2in][.1in]{} 2 & 0 & & & & & & 0 \\
\raisebox{.1in}[.2in][.1in]{} 3 & 0 & 1 & & & & & 1 = \binom{1}{1}\\
\raisebox{.1in}[.2in][.1in]{} 4 & 0 & 1 & & & & & 1 = \binom{2}{2}\\
\raisebox{.1in}[.2in][.1in]{} 5 & 0 & 1 & 2 & & & & 3 = \binom{3}{2}\\
\raisebox{.1in}[.2in][.1in]{} 6 & 0 & 1 & 3 & & & & 4 = \binom{4}{3}\\
\raisebox{.1in}[.2in][.1in]{} 7 & 0 & 1 & 4 & 5 & & & 10 = \binom{5}{3}\\
\raisebox{.1in}[.2in][.1in]{} 8 & 0 & 1 & 5 & 9 & & & 15 = \binom{6}{4}\\
\raisebox{.1in}[.2in][.1in]{} 9 & 0 & 1 & 6 & 14 & 14 & & 35 = \binom{7}{4}\\
\raisebox{.1in}[.2in][.1in]{} 10 & 0 & 1 & 7 & 20 & 28 & & 56 = \binom{8}{5}\\
\raisebox{.1in}[.2in][.1in]{} 11 & 0 & 1 & 8 & 27 & 48 & 42 & 126 = \binom{9}{5}\\
\raisebox{.1in}[.2in][.1in]{} 12 & 0 & 1 & 9 & 35 & 75 & 90 & 210 = \binom{10}{6}
\end{array}
\]
\caption{The number $\pfrk(m;d)$ of admissible pinnacle sets with maximum element $m$ and cardinality $d$.}\label{tab:admissible}
\end{table}

Notice in the table that the row sums (that is, $\sum_{d\geq 1} \pfrk(m;d)$) seem to equal
\[
 \binom{m-2}{\lfloor m/2\rfloor}.
\]
This fact always holds; we present this result now but will defer the proof until the next subsection (see Corollary~\ref{cor:count}(a)).

\begin{lemma}\label{lem:rowsums}
	For all $m \geq 1$,
	\[
		\sum_{d \geq 1} \pfrk(m;d) = \binom{m-2}{\lfloor m/2\rfloor}.
	\]
\end{lemma}

With Lemma~\ref{lem:rowsums} in hand, we can inductively compute the number of admissible pinnacle sets $S \subseteq [n]$ to be $\binom{n-1}{\lfloor (n-1)/2\rfloor}$. 
That is, if there are $\binom{n-1}{\lfloor (n-1)/2\rfloor}$ admissible pinnacle sets $S\subseteq [n]$ for some value of $n\geq 3$, then the number of admissible pinnacle sets $S \subseteq [n+1]$ would be
\begin{align*}
 1+\sum_{m=3}^{n+1} \binom{m-2}{\lfloor m/2\rfloor} &= \left(1+\sum_{m=3}^n \binom{m-2}{\lfloor m/2\rfloor}\right) + \binom{n-1}{\lfloor (n+1)/2 \rfloor},\\
 &=\binom{n-1}{\lfloor (n-1)/2 \rfloor} + \binom{n-1}{\lfloor (n+1)/2 \rfloor},\\
 &=\binom{n}{\lfloor n/2 \rfloor}.
\end{align*}

Indeed, this result is the assertion in the second half of Theorem \ref{thm:pinsets}.

The simplicity of this formula suggests a nice combinatorial explanation for the number of admissible pinnacle sets. Another nudge toward this combinatorial structure comes when we recognize that the numbers $\pfrk(m;d)$ satisfy a two-term recurrence for $m-1>2d$:
\begin{align*}
 \pfrk(m;d) &= \sum_{k<m} \pfrk(k;d-1),\\
  &=\pfrk(m-1;d-1)+\sum_{k<m-1} \pfrk(k;d-1),\\
  &=\pfrk(m-1;d-1) + \pfrk(m-1;d).
\end{align*}

In Table \ref{tab:admissible}, we see that the boundary cases for this recurrence appear to be Catalan numbers. That is,
\[\pfrk(2d+1;d) = C_d\]
for $d \geq 1$, where $C_d = \binom{2d}{d}/(d+1)$.
In fact, we prove that this observation holds in the following section.
This hints at a connection between admissible pinnacle sets and lattice paths, which we will introduce and examine more deeply in the next section. 
In particular, we will develop facts related to the enumeration of diagonal lattice paths and use them to count the number of admissible pinnacle sets.

\subsection{Diagonal Lattice Paths}

\begin{defn}
A \emph{diagonal lattice path} is a sequence of steps, composed of \emph{up-steps} $(1,1)$ and \emph{down-steps} $(1,-1)$.
\end{defn}

For fixed $n$, consider all paths from $(0,0)$ to $(n-1,1)$ if $n$ is even, or to $(n-1,0)$ if $n$ is odd.  Any such path takes $n-1$ steps, $\lfloor (n-1)/2 \rfloor$ of which are down-steps. Hence there are $\binom{n-1}{\lfloor (n-1)/2 \rfloor}$ such paths, which (we claim) is precisely the number of admissible pinnacle sets $S \subseteq [n]$. Recall that Catalan numbers count Dyck paths, i.e., diagonal lattice paths that never pass below the $x$-axis.

To convert a lattice path with $n-1$ steps into an admissible pinnacle set, we first label the steps of the path, from left to right, by $2,3,\ldots,n$. Then the labels of up-steps that are strictly below the $x$-axis and of down-steps that are weakly above the $x$-axis will form an admissible pinnacle set $S \subseteq [n]$. 

\begin{example}
The path shown in Figure~\ref{fig:diagonal lattice path and negative regions} corresponds to the set $\{4,6,12,13,19,20\}$.
\end{example}

In this section, our goal is to prove that this correspondence is a bijection between diagonal lattice paths and admissible pinnacle sets. In fact, we will refine the bijection to focus on the paths ending with a down-step, which correspond to pinnacle sets with a fixed maximum. To this end, consider diagonal lattice paths from $(0,0)$ to $(x,\epsilon_x)$ where $\epsilon_x \in \{1,2\}$ is determined by the parity of $x$.
For $x \in \mathbb{Z}$, set
\[
\epsilon_x = \begin{cases} 1 & \text{ if $x$ is odd, and}\\ 2 & \text{ if $x$ is even.}\end{cases}
\]
Clearly, appending a down-step to any such path yields a down-step that is weakly above the $x$-axis, and so in our correspondence will give a pinnacle set with maximum $x+2$.

\begin{figure}[htbp]
\begin{tikzpicture}[scale=.5]
\draw (-1,0) -- (19,0);
\draw (0,0) -- (2,-2) --node[scale=.75,midway,below right]{4} (3,-1) -- (4,-2) --node[scale=.75,midway,below right]{6} (5,-1) --(6,0) -- (7,-1) -- (10,2) --node[scale=.75,midway,above right]{12} (11,1) --node[scale=.75,midway,above right]{13} (12,0)-- (13,-1) -- (17,3) -- node[scale=.75,midway,above right]{19} (18,2)-- node[scale=.75,midway,above right]{20} (19,1);
\draw[red, ultra thick] (2,-2)--(3,-1);
\draw[red, ultra thick] (4,-2)--(5,-1);
\draw[red, ultra thick] (10,2)--(11,1);
\draw[red, ultra thick] (11,1)--(12,0);
\draw[red, ultra thick] (17,3)--(18,2);
\draw[red, ultra thick] (18,2)--(19,1);
\foreach \x in {(0,0), (1,-1), (2,-2), (3,-1), (4,-2), (5,-1), (6,0), (7,-1), (8,0), (9,1), (10,2), (11,1), (12,0), (13,-1), (14,0), (15,1), (16,2), (17,3), (18,2), (19,1)} {\fill \x circle (3pt);}
\end{tikzpicture}
\caption{A diagonal lattice path corresponding to the admissible pinnacle set $\{4,6,12,13,19,20\}$. This path has three negative regions.}\label{fig:diagonal lattice path and negative regions}
\end{figure}
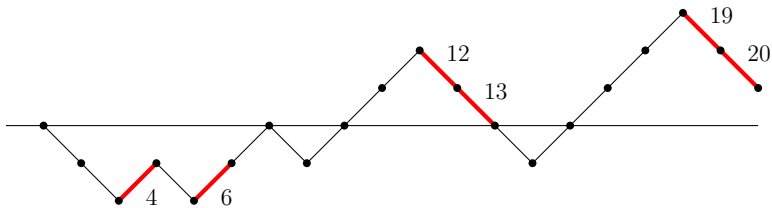

\begin{lemma}[Lattice path steps]\label{lem:number of up/down}
A diagonal lattice path from $(0,0)$ to $(x, \epsilon_x)$ has $\lfloor x/2 \rfloor + 1$ up-steps and $\lceil x/2 \rceil - 1$ down-steps.
\end{lemma}

\begin{proof}
Such a path $\Pcal$ consists of $x$ steps in total, and
\[\left|\{\text{up-steps in }\Pcal\}\right| - \left|\{\text{down-steps in }\Pcal\}\right| = \epsilon_x.\]
Thus, for odd $x$, there must be $(x+1)/2$ up-steps and $(x-1)/2$ down-steps. Similarly, for even $x$, there must be $x/2+1$ up-steps and $x/2-1$ down-steps.
\end{proof}

\begin{defn}
A \emph{negative region} in a diagonal lattice path begins with a down-step from a point $(x,0)$, terminates with an up-step to a point $(x',0)$, and does not touch the $x$-axis anywhere between those two points. The number of negative regions of a path $\Pcal$ will be denoted $\text{neg}(\Pcal)$.
\end{defn}

Figure~\ref{fig:diagonal lattice path and negative regions} depicts a diagonal lattice path $\Pcal$ for which $\text{neg}(\Pcal) = 3$.

\begin{lemma}[Sub-axis regions]\label{lem:number of neg regions}
For a diagonal lattice path $\Pcal$, 
\[\text{neg}(\Pcal) = \left| \left\{\text{down-steps in $\Pcal$ starting from the $x$-axis}\right\}\right|.\]
\end{lemma}

\begin{proof}
Negative regions can be identified uniquely by their leftmost step, which is necessarily a down-step from a point on the $x$-axis.
\end{proof}

Given a diagonal lattice path $\Pcal$, we define the \emph{marking} of $\Pcal$ to be the path obtained by marking all down-steps that are weakly above the $x$-axis and all up-steps that are strictly below the $x$-axis.
Examples of marked paths appear in Figures~\ref{fig:diagonal lattice path and negative regions} and~\ref{fig:paths and marked edges}, with marked edges colored in red.

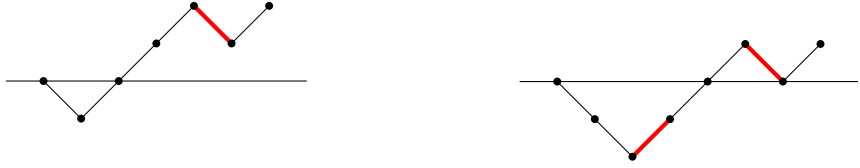
\begin{figure}[htbp]
\begin{tikzpicture}[scale=.5]
\draw (-1,0) -- (7,0);
\draw (0,0) -- (1,-1) -- (4,2) -- (5,1) -- (6,2);
\draw[red, ultra thick] (4,2) -- (5,1);
\foreach \x in {(0,0), (1,-1), (2,0), (3,1), (4,2), (5,1), (6,2)} {\fill \x circle (3pt);}
\draw[white] (0,-2) circle (3pt);
\end{tikzpicture}
\hspace{1in}
\begin{tikzpicture}[scale=.5]
\draw (-1,0) -- (8,0);
\draw (0,0) -- (2,-2) -- (5,1) -- (6,0) -- (7,1);
\draw[red, ultra thick] (2,-2) -- (3,-1);
\draw[red, ultra thick] (5,1) -- (6,0);
\foreach \x in {(0,0), (1,-1), (2,-2), (3,-1), (4,0), (5,1), (6,0), (7,1)} {\fill \x circle (3pt);}
\end{tikzpicture}
\caption{Two marked diagonal lattice paths.}
\label{fig:paths and marked edges}
\end{figure}

Now, given a marked path, we will use two sets defined in terms of its marked and unmarked edges.
Let $\Pcal$ be a marked diagonal lattice path starting at $(0,0)$ and having $x$ steps. Label the steps of the path, from left to right, by $\{2,3,\ldots,x+1\}$. Set
\begin{align*}
M(\Pcal) &= \{y : \text{the step labeled $y$ is marked}\} \cup \{x+2\}, \text{ and}\\
U(\Pcal) &= \{y : \text{the step labeled $y$ is unmarked}\} \cup \{1\}\\
&= [1,x+2] \setminus M(\Pcal).
\end{align*}

\begin{example}
For the leftmost path in Figure~\ref{fig:paths and marked edges}, $M(\Pcal) = \{6,8\}$ and $U(\Pcal) = \{1,2,3,4,5,7\}$. For the rightmost path in Figure~\ref{fig:paths and marked edges}, $M(\Pcal) = \{4,7,9\}$ and $U(\Pcal) = \{1,2,3,5,6,8\}$. 
\end{example}

It will transpire that the set $M(\Pcal)$ is a pinnacle set, and that the map $\Pcal \mapsto M(\Pcal)$ is a bijection. Before we can prove this, we elaborate on properties of diagonal lattice paths.

\begin{lemma}[Enumeration of marked edges]\label{lem:counting marked edges}
For a diagonal lattice path $\Pcal$ from $(0,0)$ to a point $(x,\epsilon_x)$,
\begin{align*}
|M(\Pcal)| &= \lceil x/2\rceil - \text{neg}(\Pcal), \text{ and}\\
|U(\Pcal)| &= \lfloor x/2\rfloor + \text{neg}(\Pcal) + 2.
\end{align*}
\end{lemma}

\begin{proof}
Two types of steps get marked in $\Pcal$: down-steps that lie weakly above the $x$-axis, and up-steps that lie strictly below the $x$-axis. Each step from $(a,b)$ to $(a+1,b+1)$ in the latter category can be paired, injectively, with the down-step from $(a',b+1)$ to $(a'+1,b)$ where $a'$ is the largest possible value less than $a$; this down-step is necessarily unmarked because it must lie below the $x$-axis. Thus the number $|M(\Pcal)| - 1$ of marked steps in $\Pcal$ is equal to the number of down-steps in $\Pcal$ that do not start at the $x$-axis. Lemmas~\ref{lem:number of up/down} and~\ref{lem:number of neg regions} complete the calculation. To compute $|U(\Pcal)|$, note that the number of steps, $x$, in $\Pcal$ is precisely $(|M(\Pcal)| - 1) + (|U(\Pcal)| - 1)$.
\end{proof}

In particular, Lemma~\ref{lem:counting marked edges} shows that $|M(\Pcal)| < |U(\Pcal)|$ for all diagonal lattice paths $\Pcal$.
The elements of the sets $M(\Pcal)$ and $U(\Pcal)$ bear some relation to each other.

\begin{lemma}[Labels of marked edges]\label{lem:marked labels are larger than unmarked}
Fix a diagonal lattice path $\Pcal$ from $(0,0)$ to a point $(x,\epsilon_x)$, and index the elements $\{m_i\}$ of $M(\Pcal)$ and $\{u_i\}$ of $U(\Pcal)$ in increasing order. Then
\[m_i > u_{i+1}\]
for all $m_i \in M(\Pcal)$.
\end{lemma}

\begin{proof}
First recall that $u_1 = 1$ by construction, and the step labels in $\Pcal$ begin with $2$. Each marked step in $\Pcal$ corresponds to a preceding (and hence smaller-labeled) unmarked step; namely, the nearest-to-the-left step of the same height. Thus $m_i > u_{i+1}$ for all $m_i \in M(\Pcal)$.
\end{proof}

\begin{prop}[Diagonal lattice paths construct admissible pinnacle sets]\label{prop:path yields admissible pinnacle set}
Fix a diagonal lattice path $\Pcal$ from $(0,0)$ to a point $(x,\epsilon_x)$. The set $M(\Pcal)$ is an admissible pinnacle set. 
\end{prop}

\begin{proof}
Index the elements $\{m_i\}$ of $M(\Pcal)$ and $\{u_i\}$ of $U(\Pcal)$ in increasing order and consider the permutation
\[u_1m_1u_2m_2u_3m_3u_4 \cdots u_dm_du_{d+1}u_{d+2}\cdots\]
where $d= |M(\Pcal)|$.
By Lemma~\ref{lem:marked labels are larger than unmarked}, $m_i > u_{i+1}$. Moreover, the elements of $U(\Pcal)$ are indexed in increasing order, so $u_{i+1} > u_i$, and $m_i > u_i$ by transitivity. Thus the pinnacle set of this permutation is exactly $M(\Pcal)$.
\end{proof}

\begin{example}
For the leftmost path in Figure~\ref{fig:paths and marked edges}, the permutation produced by Proposition~\ref{prop:path yields admissible pinnacle set} is $16283457 \in S_8$. For the rightmost path in Figure~\ref{fig:paths and marked edges}, the permutation is $142739568 \in S_9$.
\end{example}

We now show that the mapping from diagonal lattice paths to pinnacle sets, described in Proposition~\ref{prop:path yields admissible pinnacle set}, is invertible. Note that the pinnacle set described in Proposition~\ref{prop:path yields admissible pinnacle set} has size $\lceil x/2\rceil - \text{neg}(\Pcal)$, by Lemma~\ref{lem:counting marked edges}, and its maximum value is $x+2$. We will show that we can start with an arbitrary pinnacle set, of size $\lceil x/2\rceil - \text{neg}(\Pcal)$ and having maximum value $x+2$, and produce the corresponding diagonal lattice path from $(0,0)$ to the  point $(x,\epsilon_x)$.

\begin{defn}\label{defn:constructing the path from the pinnacle set}
Let $S$ be an admissible pinnacle set with $\max S = m$. 
Define the diagonal lattice path $\Pcal(S)$ as follows.
\begin{quote}
Start at the point $(x_0,y_0) := (m-2,\epsilon_m)$, with $S_0 := S$.\\
Set $S_1 := S_0 \setminus \{m\}$. \\
For $i$ from $1$ to $m-2$: \\ 
\hspace*{.3in} If $\max S_i = m - i$ then set $S_{i+1} := S_i \setminus \{m-i\}$ and:\\
\hspace*{.6in} If $y_{i-1} \ge 0$, then set $(x_i,y_i) := (x_{i-1} - 1, y_{i-1} +1)$.\\
\hspace*{.6in} Otherwise (that is, if $y_{i-1} < 0$), set $(x_i,y_i) := (x_{i-1} - 1, y_{i-1} -1)$.\\
\hspace*{.3in} Otherwise (that is, if $\max S_i \neq m - i$), then set $S_{i+1} := S_i$ and:\\
\hspace*{.6in} If $y_{i-1} \ge 0$, then set $(x_i,y_i) := (x_{i-1} - 1, y_{i-1} -1)$.\\
\hspace*{.6in} Otherwise (that is, if $y_{i-1} < 0$), set $(x_i,y_i) := (x_{i-1} - 1, y_{i-1} +1)$.
\end{quote}
\end{defn}

Consider the admissible pinnacle set $S = \{4,7,9\}$, for which $m = \max S = 9$. The procedure described in Definition~\ref{defn:constructing the path from the pinnacle set} produces the following data.
\[\begin{array}{c||c|c|c|c|c|c|c|c}
i & 0 & 1 & 2 & 3 & 4 & 5 & 6 & 7\\
\hline
\raisebox{.1in}[.2in][.1in]{} S_i & \{4,7,9\} & \{4,7\} & \{4,7\} & \{4\} & \{4\} & \{4\} & \emptyset & \emptyset \\
\hline
\raisebox{.1in}[.2in][.1in]{} (x_i, y_i) & (7,1) & (6,0) & (5,1) & (4,0) & (3,-1) & (2,-2) & (1,-1) & (0,0)
\end{array}\]
The path described by this data is the rightmost path depicted in Figure~\ref{fig:paths and marked edges}.

We will show that this map $S \mapsto \Pcal(S)$, from admissible pinnacle sets to paths, is the inverse of the map $\Pcal \mapsto M(\Pcal)$. First, however, we must show that the diagonal lattice path $\Pcal(S)$ of Definition~\ref{defn:constructing the path from the pinnacle set} is, in fact, the kind of path we want to work with; namely, that its left endpoint is $(0,0)$.

\begin{lemma}[Endpoint of pinnacle-created paths] 
For any admissible pinnacle set $S$, the left endpoint of the diagonal lattice path $\Pcal(S)$ is $(0,0)$.
\end{lemma}

\begin{proof}
That the leftmost endpoint of $\Pcal(S)$ has $x$-coordinate $0$ is clear by construction. Now consider the $y$-coordinate of this point.

Let $\max S = m$. The path $\Pcal(S)$ has $m-2$ steps, constructed in Definition~\ref{defn:constructing the path from the pinnacle set} as $i$ ranges from $1$ to $m-2$. Recall from Proposition~\ref{prop:admissible} that $m \ge 2|S| + 1$.

Consider the right-to-left path construction described in Definition~\ref{defn:constructing the path from the pinnacle set}. Each element of $S \setminus \{m\}$ moves the path away from the line $y = -0.5$, whereas each element of $[2,m-1]\setminus S$ moves the path toward (and, if $y_{i-1} \in \{-1,0\}$, across) that line. We have an excess of steps moving toward this line because
\[ |[2,m-1] \setminus S| = |[2,m] \setminus S| \geq |S|  > |S \setminus \{m\}|,\]
so the path $\Pcal(S)$ terminates at $(0,y)$ with $y \in \{-1,0\}$ if $\epsilon_m = 1$, or $y \in \{-1,0,1\}$ if $\epsilon_m = 2$.

If $\epsilon_m = 2$, then $m$ is even and $\Pcal(S)$ is a path with an even number of steps. Thus the heights of its endpoints have the same parity. On the other hand, if $\epsilon_m = 1$, then $m$ is odd and $\Pcal(S)$ is a path with an odd number of steps, meaning that the heights of its endpoints have opposite parities. In either case, the leftmost height of $\Pcal(S)$ must be even, and the only available option is to land on the $x$-axis itself.
\end{proof}

We can now prove that the two maps discussed above, between pinnacle sets and diagonal lattice paths, are inverse of each other.

\begin{thm}[Bijection between admissible pinnacle sets and diagonal lattice paths]\label{thm:admissible pinnacle set yields path}
The map $S \mapsto \Pcal(S)$ from admissible pinnacle sets to diagonal lattice paths is the inverse of the map $\Pcal \mapsto M(\Pcal)$, and together these maps give a bijection between admissible pinnacle sets and diagonal lattice paths.
\end{thm}

\begin{proof}
Let $S$ be an admissible pinnacle set. By the construction given in Definition~\ref{defn:constructing the path from the pinnacle set}, elements of $S \setminus \{\max S\}$ correspond to down-steps that are weakly above the $x$-axis and up-steps that are strictly below the $x$-axis in the resulting path (that is, steps that move away from the line $y = -0.5$). These are exactly the steps in a path that are marked by a lattice path marking, and which, together with $\max S$, constitute the set $M(\Pcal(S))$.
\end{proof}


We are now ready to prove the main result of this section.

\begin{thm}[Enumerating admissible pinnacle sets in terms of paths]\label{thm:path bijection}
For all $m,d \ge 1$, the number $\pfrk(m;d)$ of admissible pinnacle sets with maximum element $m$ and cardinality $d$ is
\[\pfrk(m;d) = \left|\left\{\begin{matrix} 
\text{diagonal lattice paths $\Pcal$ from $(0,0)$ to $(m-2,\epsilon_m)$},\\
\text{with } \text{neg}(\Pcal) = \lceil m/2\rceil - 1 - d \end{matrix}
\right\}\right|.\]
\end{thm}

\begin{proof}
Proposition~\ref{prop:path yields admissible pinnacle set} and Theorem~\ref{thm:admissible pinnacle set yields path} give a bijection between admissible pinnacle sets with maximum element $m$ and diagonal lattice paths from $(0,0)$ to $(m-2,\epsilon_m)$. Let $S$ and $\Pcal$ be such a corresponding pair. By Lemma~\ref{lem:counting marked edges}, 
\begin{align*}
|S| &= |M(\Pcal)|\\
&= \lceil (m-2)/2 \rceil - \text{neg}(\Pcal)\\
&= \lceil m/2 \rceil - 1 - \text{neg}(\Pcal),
\end{align*}
which completes the proof.
\end{proof}

We now pause to demonstrate the bijection of Theorem~\ref{thm:admissible pinnacle set yields path}.

\begin{example}
The leftmost lattice path in Figure~\ref{fig:paths and marked edges} corresponds to the admissible pinnacle set $\{6,8\}$, counted by $\pfrk(8;2)$, and the rightmost path corresponds to the admissible pinnacle set $\{4,7,9\}$, counted by $\pfrk(9;3)$.
\end{example}

Because it is easy to count diagonal lattice paths between two fixed points, we can make the following enumerative corollaries. The former of these establishes the second half of Theorem~\ref{thm:pinsets} from Section~\ref{sec:intro}.
Meanwhile, the latter of these establishes the boundary case, discussed above, in the recursive expression for $\pfrk(m;d)$ when $m \geq 2d + 2$.

\begin{cor}[Enumerating admissible pinnacle sets] \label{cor:count} \
\begin{enumerate}\renewcommand{\labelenumi}{(\alph{enumi})}
\item The total number of admissible pinnacle sets (regardless of size) with maximum element $m$ is
\[\binom{m-2}{\lfloor m/2\rfloor}.\]
\item For $m=2d+1$, the number of admissible pinnacle sets with maximum element $m$ and size $d$ is the Catalan number $C_d = \binom{2d}{d}/(d+1)$.
\end{enumerate}
\end{cor}

\begin{proof}
\ \begin{enumerate}\renewcommand{\labelenumi}{(\alph{enumi})}
\item This is simply the total number of diagonal lattice paths from $(0,0)$ to $(m-2, \epsilon_m)$. Each contains $m-2$ steps, of which $\lfloor m/2 \rfloor$ are up-steps, by Lemma~\ref{lem:number of up/down}.
\item By Theorem~\ref{thm:path bijection}, this is the number of diagonal lattice paths from $(0,0)$ to $(m-2, 1)$ that never go below the $x$-axis. Since every Dyck path must end with a down-step, these are in bijective correspondence with Dyck paths from $(0,0)$ to $(m-1,0)$, and Dyck paths are enumerated by the Catalan numbers. \qedhere
\end{enumerate}
\end{proof}

\begin{rem}[Alternate way to count admissible sets]\label{rem:alternate}
Bruce Sagan has suggested an alternate way to prove there are $\binom{n}{\lfloor n/2\rfloor}$ admissible sets $S\subseteq [n+1]$, using only induction and the boundary case mentioned in Corollary \ref{cor:count}(b). 

Consider the number of admissible sets with $\max S < n+1$. By induction these are counted by $\binom{n-1}{\lfloor (n-1)/2 \rfloor}$. Thus it suffices to show that the number of admissible sets with $\max S = n+1$ is given by
\[
\binom{n-1}{\lfloor (n+1)/2 \rfloor}=
 \begin{cases} 
 \binom{2k-1}{k} & \mbox{if } n=2k, \\
 \binom{2k-2}{k} & \mbox{if } n=2k-1.
 \end{cases}
\]

If $n=2k$, then for every admissible set $T$ with $\max T \leq n$, we claim the set $S=T\cup \{n+1\}$ is also admissible. To see this, consider the canonical permutation $w_T \in S_n$ from Equation \eqref{eqn:canonical perm with given pinnacle set}. Since $n$ is even, $w_T$ must end with an ascent. By inserting $n+1$ in this ascent position, we form a new pinnacle. Thus $S=T\cup \{n+1\}$ is admissible. Since $\binom{2k-1}{k} = \binom{2k-1}{k-1} = \binom{n-1}{\lfloor (n-1)/2\rfloor}$, the result follows.

Now suppose $n=2k-1$. The only way the previous construction (inserting $n+1$ in $w_T$) will not work is if $|T|=k-1$ is as large as possible. This is because if $\max S = n+1$, then $|S|<(n+1)/2=k$. Therefore, the number of viable such $T$ are those for which $|T|<k-1$. By induction, we have a total of $\binom{n-1}{\lfloor (n-1)/2\rfloor} = \binom{2k-2}{k-1}$ admissible $T \subseteq [n]$, so those with $|T|<k-1$ are given by
\[
 \binom{2k-2}{k-1} - |\{ T \mbox{ admissible} : \max T \leq 2k-1, |T|=k-1\}|. 
\]
Using the fact that the extremal admissible sets are counted by Catalan numbers (Corollary \ref{cor:count}(b)), this becomes
\[
 \binom{2k-2}{k-1} - \frac{1}{k}\binom{2k-2}{k-1}  = \binom{2k-2}{k},
\]
as desired.
\end{rem}

\section{Recurrences, explicit formulas, and bounds for $p_S(n)$}\label{sec:rec}

Now that we have characterized and enumerated admissible pinnacle sets, we turn to the question of counting permutations with a given pinnacle set. Recall that $p_S(n)$ denotes the number of permutations $w \in S_n$ with $\Pin(w) = S$.

To begin our study of $p_S(n)$, we make the easy observation that there are $2^{n-1}$ permutations in $S_n$ having no peaks; that is,
\[
 p_{\emptyset}(n) = 2^{n-1}.
\]
Indeed, if $\Pin(w)=\emptyset$, then we can write $w = u1v$, a concatenation of strings, where $u$ is a word whose letters are strictly decreasing and $v$ is a word whose letters are strictly increasing. If $w \in S_n$, then each such permutation is determined by the elements of $u$, which can be any subset of the $(n-1)$-element set $\{2,3,\ldots,n\}$.

A similar argument shows that when $S$ is nonempty, we can reduce to the case where $w \in S_t$ for any $t \in [\max S,n]$, because none of the letters $\{t+1,\ldots, n\}$ are pinnacles in $w$.

\begin{lemma}[Reduction of permutation size]\label{lem:null/reduction}
If $S$ is nonempty and $t \in [\max S,n]$, then
\[
 p_S(n) = 2^{n-t}p_S(t).
\]
For permutations with no pinnacles (nor peaks), we have $p_{\emptyset}(n) = 2^{n-1}$. 
\end{lemma}

\begin{proof}
We prove only the first statement, as the case for $S = \emptyset$ was discussed above.

Suppose that $w \in S_n$ and $\Pin(w)=S$. Further suppose that $t \in [\max S,n]$. Because none of the letters $\{t+1,\ldots,n\}$ are pinnacles in $w$, we can write $w = uw'v$, a concatenation of strings, for some $w' \in S_t$ with $\Pin(w') = S$. Since the elements of $u$ and $v$ are drawn from the set $[n]\setminus [t]$, it must be the case that $u$ is a decreasing word and $v$ is increasing. Hence $w$ depends only on $w'$ and the set of elements in $u$. The set of elements in $u$ can be any subset of $[n]\setminus [t]$, yielding $2^{n-t}$ possibilities. The number of permutations $w' \in S_t$ having pinnacle set $S$ is, by definition, $p_S(t)$, and so $p_S(n) = 2^{n-t}p_S(t)$.
\end{proof}

In practice, we will most often employ Lemma~\ref{lem:null/reduction} with $t = \max S$ or $t = n-1$.

\subsection{A quadratic recurrence}\label{sec:quad}

Let us assume that $S$ is a nonempty admissible pinnacle set with $\max S = n$. To construct one of the permutations in $S_n$ counted by $p_S(n)$, we could first choose the elements that will appear to the left of $n$ and those that will appear to the right of $n$, and then try to arrange the letters on each side of $n$ in order to achieve our desired pinnacle set. To be more precise consider the following steps:
\begin{enumerate}
	\item Write $[n-1] = A \sqcup A^c$ as a disjoint union of nonempty sets.
	\item Let $I=S\cap A$ (pinnacles to appear to the left of $n$) and $J = S \cap A^c$ (pinnacles to appear to the right of $n$).
	\item If possible, form permutations $u$ of the set $A$ and $v$ of the set $A^c$, with $\Pin(u) = I$ and $\Pin(v) = J$.
	\item Let $w = u\,n\,v$, a concatenation of strings. Then $\Pin(w) = I\cup \{n\}\cup J =S$.
\end{enumerate}

We will now analyze the number of ways to perform this procedure.

\begin{defn}
The \emph{standardization map} relative to a set $X = \{x_1 < x_2 < \cdots\}$ is
\[\std_X(x_i) =i.\]
\end{defn}

Fix a nonempty set $A=\{ a_1 < a_2 < \cdots < a_{|A|}\} \subsetneq [n-1]$, and let
\[
I=\std_A(S) = \{ i : a_i \in S\}.
\]
In other words, $I$ is the set of relative values of pinnacles within the subset $A$. 

With this notation, the number of permutations $u$ of set $A$ such that $\Pin(u) = S'$ equals the number of permutations in $S_{|A|}$ with pinnacle set $I$. That is, the number of such $u$ is $p_I(|A|)$. Likewise, letting $J=\std_{A^c}(S)$ denote the set of relative values of the pinnacles within $A^c$, we have $p_J(|A^c|) = p_J(n-1-|A|)$ ways to form the permutation $v$.

Running over all cases of the set $A$, we get the following result.

\begin{prop}[The quadratic recurrence]\label{prp:quadratic}
Suppose that $S$ is a nonempty admissible pinnacle set with $\max S = n$. Then
\begin{equation}\label{eq:quad}
 p_S(n) = \sum_{\emptyset \neq A \subsetneq [n-1]} p_{\std_A(S)}(|A|) \cdot p_{\std_{A^c}(S)}(n-1-|A|).
\end{equation}
\end{prop}

This construction is illustrated in Figure \ref{fig:quadratic}, and we give a specific example below. We remark that the recursive structure inherent in the quadratic recurrence suggests that there might be a relationship between pinnacle sets and permutation pattern containment, perhaps for vincular patterns in particular. Indeed, a peak is exactly a $\underbracket[.5pt][1pt]{132}$ or $\underbracket[.5pt][1pt]{231}$ vincular pattern.

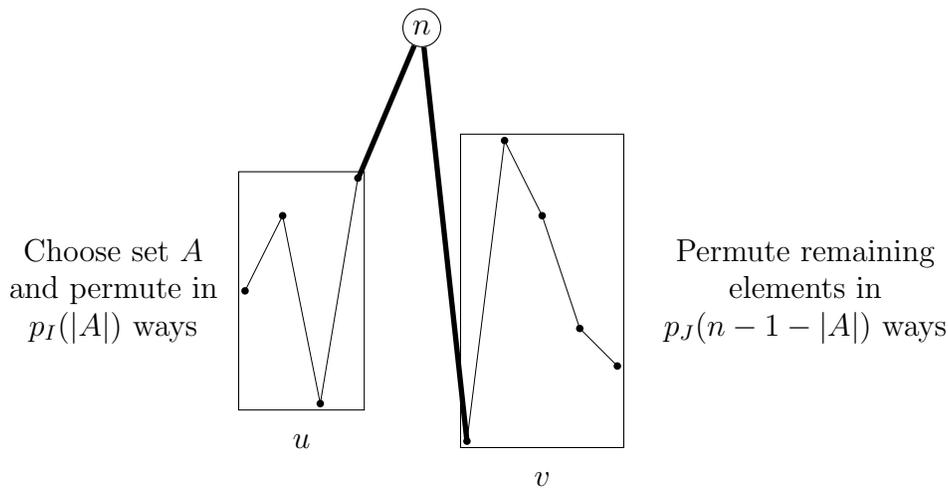
\begin{figure}[htbp]
\[
\begin{tikzpicture}[>=stealth,bend angle=45,auto,xscale=.4,yscale=.5]
\tikzstyle{state}=[circle,draw=black,inner sep=2]
\draw (0,10) node[state] (1) {$n$};
\draw (-4,3) node[rectangle, draw=black, inner sep=2,scale=.5] (2) {
\begin{tikzpicture}
\draw (0,4) node[fill=black,circle, inner sep =2] {} -- (1,6) node[fill=black,circle, inner sep =2] {} -- (2,1) node[fill=black,circle, inner sep =2] {} -- (3,7) node[fill=black,circle, inner sep =2] {}; 
\end{tikzpicture}
};
\draw[line width=2] (1)--(-2.11,6);
\draw[line width=2] (1)--(1.5,-1);
\draw (4,3) node[rectangle, draw=black, inner sep=2,scale=.5] (3) {
\begin{tikzpicture}
\draw (0,2) node[fill=black,circle, inner sep =2] {} -- (1,10) node[fill=black,circle, inner sep =2] {} -- (2,8) node[fill=black,circle, inner sep =2] {} -- (3,5) node[fill=black,circle, inner sep =2] {}-- (4,4) node[fill=black,circle, inner sep =2] {}; 
\end{tikzpicture}
};
\draw (2) node[xshift=-2.5cm] {\begin{tabular}{c} Choose set $A$ \\and permute in \\$p_I(|A|)$ ways\end{tabular}};
\draw (3) node[xshift=3.5cm] {\begin{tabular}{c} Permute remaining \\elements in \\$p_J(n-1-|A|)$ ways\end{tabular}};
\draw (2) node[yshift=-2cm] {$u$};
\draw (3) node[yshift=-2.5cm] {$v$};
\end{tikzpicture}
\]
\caption{Construction of the quadratic recurrence.}\label{fig:quadratic}
\end{figure}

\begin{example}\label{ex:quadratic recurrence}
Let $n =9$ and $S = \{4,7,9\}$. Then we would choose any proper nonempty subset of $[8]$, say $A=\{1,2,4\}$, so that $A^c = \{ 3,5,6,7,8\}$. Here,
\[
 I = \std_{\{1,2,4\}}(\{4,7,9\}) = \{3\},
\]
while 
\[
 J=\std_{\{3,5,6,7,8\}}(\{4,7,9\}) = \{4\},
\]
so this $A$ contributes a term of $p_{\{3\}}(3)p_{\{4\}}(5) = 2 \cdot 24 = 48$ to the computation of $p_{\{4,7,9\}}(9)$.
\end{example}

While it may seem that the quadratic recurrence must sum over $2^{n-1}-2$ subsets $A$, note that many of these selections contribute zero to the sum, because both $\std_A(S)$ and $\std_{A^c}(S)$ must themselves be admissible pinnacle sets.

\begin{example}
With the set $S = \{4,7,9\}$ of Example~\ref{ex:quadratic recurrence}, only $44$ of the possible $2^8-2 = 254$ summands in Equation~\eqref{eq:quad} are nonzero.
\end{example}

By combining Proposition \ref{prp:quadratic} with Lemma~\ref{lem:null/reduction}, we can obtain explicit formulas for pinnacle sets with one or two elements.

\begin{prop}\label{prp:12peak}
We have the following explicit formulas for admissible pinnacle sets with one or two elements. Let $3\leq l < m$. Then, for any $n\geq l$,
\begin{equation}\label{eq:onepeak}
 p_{\{l\}}(n) = 2^{n-2}(2^{l-2}-1)
\end{equation}
 and for any $n\geq m$,
\begin{equation}\label{eq:2peak}
p_{\{l,m\}}(n) = 2^{n+m-l-5}\left(3^{l-1}-2^l + 1\right) - 2^{n-3}(2^{l-2}-1).
\end{equation}
\end{prop}

\begin{proof}
First consider a pinnacle set with one element, say $S = \{l\}$. Then Equation~\eqref{eq:quad} tells us that each nonempty set $A \subsetneq [l-1]$ contributes 
\[
p_{\emptyset}(|A|)p_{\emptyset}(l-1-|A|) = 2^{|A|-1}2^{l-2-|A|} = 2^{l-3},
\]
to the sum. As there are $2^{l-1}-2$ subsets $A$ to consider, we find that $p_{\{l\}}(l) = 2^{l-3}(2^{l-1}-2)=2^{l-2}(2^{l-2}-1)$. By Lemma~\ref{lem:null/reduction}, we see that for any $n\geq l \geq 3$, the number of permutations in $S_n$ with pinnacle set $\{l\}$ is
\[
 p_{\{l\}}(n) = 2^{n-2}(2^{l-2}-1),
\]
which proves Equation \eqref{eq:onepeak}.

Now to prove Equation \eqref{eq:2peak}, we suppose $w$ in $S_m$ with pinnacle set $\{l,m\}$, where $l < m$. We now analyze all sets $A$ that contribute to the sum of Equation \eqref{eq:quad}. 

First of all, notice we can count all possibilities where $l$ appears
to the left of $m$, i.e., where $l \in A$, and multiply by two. Thus, we assume $l\in A$ for the time being.

In order for set $A$ to form a permutation whose only pinnacle is $l$, i.e., for $I=\std_{A}(S)$ to be admissible, $A$ must contain at least two elements smaller than $l$. Let $j\geq 2$ denote the number of elements in $A$ smaller than $l$, so that $I = \{j+1\}$. Further, let $k$ denote the number of elements in $A$ that are bigger than $l$ and smaller than $m$.  Since $A^c$ is nonempty, we must have $0<m-1-|A|= m-2-j-k$, or $j+k < m-2$.

Given fixed $j \geq 2$ and $k$ as above, the number of ways to permute set $A$
to get a pinnacle set of $\{l\}$ is, by Equation \eqref{eq:onepeak},
\begin{align*}
p_{\{j+1\}}(|A|) &= p_{\{j+1\}}(k+j+1)\\
 &= 2^k p_{\{j+1\}}(j+1)\\
 &= 2^{k+j-1}(2^{j-1}-1).
\end{align*}

Since we don't want any pinnacles on the other side of $m$, i.e., since $J = \std_{A^c}(S) = \emptyset$, there are
\[
p_{\emptyset}(m-1-|A|) = 2^{m-3-j-k}
\]
ways to permute the elements on the right side of $m$.

Therefore the total contribution from set $A$ is
\begin{align*}
p_{\{j+1\}}(|A|)p_{\emptyset}(m-1-|A|) &=2^{m-3-j-k}2^{k+j-1}(2^{j-1}-1) \\
 &= 2^{m-4}(2^{j-1}-1).
\end{align*}
Notice that all that really matters here is $j$ (and not $k$).

It remains to describe how to count sets $A$ with these properties. First, there are $\binom{l-1}{j}$ ways to choose $j$ elements smaller than $l$. There are $\binom{m-1-l}{k}$ ways to choose $k$ elements greater than $l$ and less than $m$. Thus, summing over all $j$ and $k$ (and doubling to consider the possibility that $l\notin A$), we find
\[
p_{\{l,m\}}(m) = 2^{m-3}\sum_{\substack{ 2 \leq j \leq l-1 \\ 0\leq k\leq m-l-1\\ j+k < m-2}}\binom{l-1}{j}\binom{m-1-l}{k}(2^{j-1}-1).
\]
The condition that $j+k < m-2$ excludes only the case that $j=l-1$ and $k=m-1-l$, i.e., the case that $A^c$ is empty.

This means we can write
\begin{align*}
 p_{\{l,m\}}(m) &= 2^{m-3}\sum_{2 \leq j \leq l-1} \binom{l-1}{j}(2^{j-1}-1)\sum_{0\leq k\leq m-1-l} \binom{m-1-l}{k} - 2^{m-3}(2^{l-2}-1),\\
  &=2^{m-3}\sum_{2 \leq j \leq l-1} \binom{l-1}{j}(2^{j-1}-1)\cdot 2^{m-1-l} - 2^{m-3}(2^{l-2}-1),\\
  &=2^{2m-l-4}\sum_{2 \leq j \leq l-1} \binom{l-1}{j}(2^{j-1}-1) - 2^{m-3}(2^{l-2}-1).
\end{align*}
A bit of manipulation shows
\begin{align*}
2\cdot\sum_{2 \leq j \leq l-1} \binom{l-1}{j}(2^{j-1}-1) &= 1+\sum_{0\leq j\leq l-1} \binom{l-1}{j}(2^j-2)\\
&= 1 + \sum_{0\leq j\leq l-1}\binom{l-1}{j}2^j - 2\sum_{0\leq j \leq l-1} \binom{l-1}{j},\\
&=1 + 3^{l-1} - 2^l.
\end{align*}
Thus,
\[
p_{\{l,m\}}(m) =2^{2m-l-5}(3^{l-1}-2^l+1) - 2^{m-3}(2^{l-2}-1).
\]
Applying Lemma \ref{lem:null/reduction} yields \eqref{eq:2peak}, completing the proof.
\end{proof}

There may also be other special cases of explicit formulas that one can deduce from the quadratic recurrence, by exploring precisely which nonzero terms appear in the sum. For now, though, we turn to another recursive approach.

\subsection{A linear recurrence}

In this section, we present a different way to build from the case of a one-element pinnacle set to that of a two-element set. As before, suppose that $S=\{l,m\}$ with $l < m$.

Consider some $w \in S_m$ for which $\Pin(w) = \{l,m\}$, and let $w' \in S_{m-1}$ be the permutation obtained by deleting the letter $m$ from $w$. Then either $\Pin(w') = \{l\}$, or $\Pin(w') = \{j,l\}$ where $j$ was adjacent to $m$ in $w$. Thus, to evaluate $p_{\{l,m\}}(m)$, we should count such $w' \in S_{m-1}$ and the ways to insert $m$ appropriately. More precisely, we want permutations $u \in S_{m-1}$ with exactly one pinnacle, $\Pin(u) = l$, and permutations $v \in S_{m-1}$ with exactly two pinnacles, $\{j,l\}$.

Let $u \in S_{m-1}$ be a permutation with $\Pin(u) = \{l\}$. We want to insert the letter $m$ into $u$ to produce a permutation $w \in S_m$ having pinnacle set $S = \{l,m\}$. We cannot insert $m$ at either end of $u$ (because then $m$ would not be a pinnacle of $w$), nor on either side of $l$ in $u$ (because then $l$ would not be a pinnacle of $w$). Because $l$ is a pinnacle of $u$, this letter $l$ cannot appear at either end of the word $u$. Thus there are $m-4$ positions at which inserting $m$ into $u \in S_{m-1}$ will yield a permutation in $S_m$ having pinnacle set $\{l,m\}$. (This is depicted in Figure \ref{fig:insertl}.) The permutations constructed in this manner contribute
\[
 (m-4)p_{\{l\}}(m-1)
\]
to the count $p_{\{l,m\}}(m)$.

\begin{figure}[htbp]
\[
\begin{tikzpicture}[>=stealth,bend angle=45,auto,xscale=.4,yscale=.5]
\tikzstyle{state}=[circle,draw=black,inner sep=2]
\draw (1,10) node[draw=none] (1) {};
\draw (4,8) node[state] (2) {};
\draw (6,6) node[state] (3) {};
\draw (10,2) node[state] (4) {};
\draw (15,7) node[state] (5) {$l$};
\draw (18,4) node[state] (6) {};
\draw (21,1) node[state] (7) {};
\draw (23,3) node[state] (8) {};
\draw (25,5) node[state] (9) {};
\draw (29,9) node[state] (10) {};
\draw (31,10) node[draw=none] (11) {};
\draw (5,10) node[state] (m1) {$m$};
\draw (8,10) node[state] (m2) {$m$};
\draw (19.5,10) node[state] (m3) {$m$};
\draw (22,10) node[state] (m4) {$m$};
\draw (24,10) node[state] (m5) {$m$};
\draw (27,10) node[state] (m6) {$m$};
\draw[line width=2] (1)--(2);
\draw (2)--(3);
\draw (3)--(4);
\draw[line width=2] (5)--(6);
\draw (6)--(7);
\draw[line width=2] (4)--(5);
\draw (7)--(8);
\draw (8)--(9);
\draw (9)--(10);
\draw[line width=2] (10)--(11);
\draw[dashed] (2)--(m1)--(3);
\draw[dashed] (4)--(m2)--(3);
\draw[dashed] (6)--(m3)--(7);
\draw[dashed] (7)--(m4)--(8);
\draw[dashed] (8)--(m5)--(9);
\draw[dashed] (9)--(m6)--(10);
\end{tikzpicture}
\]
\caption{Insert a new highest peak in any of the gaps except those on the far left, far right, and adjacent to an existing peak.}\label{fig:insertl}
\end{figure}
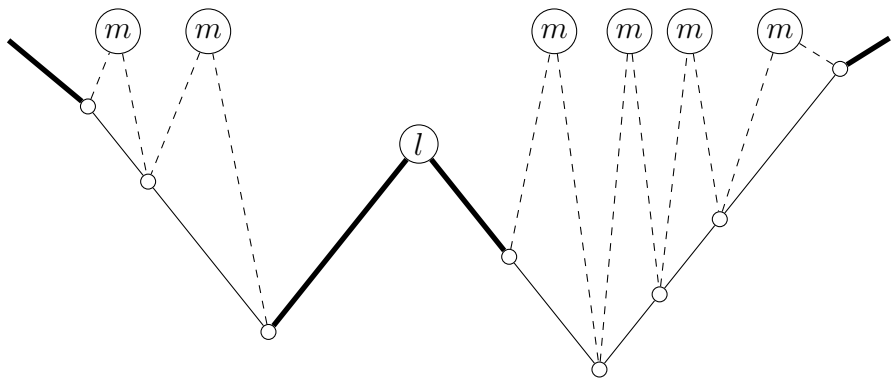

Now suppose that $v \in S_{m-1}$ is a permutation with pinnacle set $\Pin(v) = \{j,l\}$, where $l\neq j < m$. In this situation, if we place $m$ immediately to the left or right of $j$, then $j$ is no longer a pinnacle, but both $l$ and $m$ are pinnacles. Thus for each admissible pinnacle set $\{j,l\}$ with $j < m$, we have a contribution of $2p_{\{j,l\}}(m-1)$ as well.

Hence, applying Equation~\eqref{eq:onepeak} we get 
\begin{equation}\label{eq:twopeak}
 p_{\{l,m\}}(m) = (m-4)2^{m-3}(2^{l-2}-1) + 2\sum_{l\neq j < m} p_{\{j,l\}}(m-1).
\end{equation}

This line of reasoning can be generalized to sets $S=\{ s_1 < s_2 < \cdots < s_d\}$, with $s_d=m$. The analysis proceeds along the same steps as in the case $d=2$, which produced Equation~\eqref{eq:twopeak}, and applying Lemma~\ref{lem:null/reduction}.

\begin{prop}[A linear recurrence]\label{prp:linear}
Suppose that $S$ is an admissible pinnacle set with $|S|=d$ and $\max S =m$. Then for any $n\geq m$,
\begin{equation}
 p_S(n) = 2^{n-m}\left( (m-2d) p_{S \setminus \{m\}}(m-1) + 2\sum_{\substack{T = (S \setminus \{m\}) \cup \{j\} \\ j \in [m] \setminus S}} p_T(m-1) \right).  \label{eq:rec2}
\end{equation}
\end{prop}

\begin{proof}
When deleting $m$ from a permutation $w$ with $\Pin(w)=S$, either we reduce the number of peaks by one (i.e., we have $u$ such that $\Pin(u) = S\setminus \{m\}$) or the resulting permutation has the same number of peaks (i.e., we have $v$ such that $\Pin(v) = (S\setminus\{m\}) \cup \{j\}$ for some $j < m$).

First, suppose that $u \in S_{m-1}$ is any permutation with $\Pin(u) = S \setminus \{m\}$, and insert $m$ into a gap of $u$ to form a permutation with pinnacle set $S$, as in Figure~\ref{fig:insertl}. The forbidden gaps are those at the far left end of $u$, at the far right end of $u$, and adjacent to any of the existing peaks. Since $u$ has $m-1$ letters, there are $m-2$ internal gaps, and since $u$ has $d-1$ peaks, we must avoid $2(d-1)$ of these. This leaves
\[
 (m-2)-2(d-1) = m - 2d
\]
gaps in which we can place $m$ to obtain a permutation $w \in S_{m}$ with $\Pin(w) = S$. In other words, the permutations constructed in this manner contribute
\[
(m - 2d)p_{S \setminus \{m\}}(m-1)
\]
to the count $p_S(m)$.

Next, suppose that $T = (S \setminus \{m\}) \cup \{j\}$ for some $j \in [m] \setminus S$. Let $v \in S_{m-1}$ have $\Pin(v) = T$. Then we can form a permutation with pinnacle set $S$ by inserting $m$ to the left or to the right of the letter $j$. This will mean that $j$ no longer sits at a peak, but $m$ does, as shown in Figure \ref{fig:deletej}.

Combining the two cases produces
\[
 p_S(m) = (m-2d) p_{S\setminus \{m\}}(m-1) + 2\sum_T p_T(m-1),
\]
where the sum is over all $T$ of the form $T=(S\setminus\{m\}) \cup \{j\}$ for some $j \in [m] \setminus S$. Lemma~\ref{lem:null/reduction} completes the proof.
\end{proof}

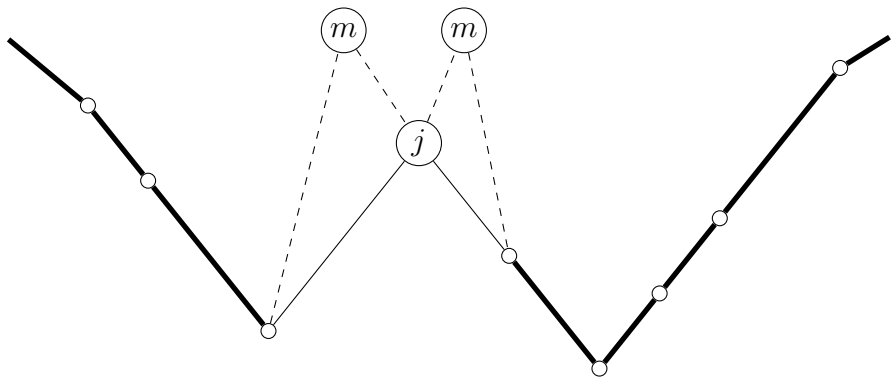
\begin{figure}[htbp]
\[
\begin{tikzpicture}[>=stealth,bend angle=45,auto,xscale=.4,yscale=.5]
\tikzstyle{state}=[circle,draw=black,inner sep=2]
\draw (1,10) node (1) {};
\draw (4,8) node[state] (2) {};
\draw (6,6) node[state] (3) {};
\draw (10,2) node[state] (4) {};
\draw (15,7) node[state] (5) {$j$};
\draw (18,4) node[state] (6) {};
\draw (21,1) node[state] (7) {};
\draw (23,3) node[state] (8) {};
\draw (25,5) node[state] (9) {};
\draw (29,9) node[state] (10) {};
\draw (31,10) node (11) {};
\draw (12.5,10) node[state] (m1) {$m$};
\draw (16.5,10) node[state] (m2) {$m$};
\draw[line width=2] (1)--(2);
\draw[line width=2] (2)--(3);
\draw[line width=2] (3)--(4);
\draw (5)--(6);
\draw[line width=2] (6)--(7);
\draw (4)--(5);
\draw[line width=2] (7)--(8);
\draw[line width=2] (8)--(9);
\draw[line width=2] (9)--(10);
\draw[line width=2] (10)--(11);
\draw[dashed] (4)--(m1)--(5);
\draw[dashed] (5)--(m2)--(6);
\end{tikzpicture}
\]
\caption{Inserting a new highest peak adjacent to an existing peak replaces that element of the pinnacle set. The element $j$ was in the original pinnacle set, but now it is replaced by $m$.}\label{fig:deletej}
\end{figure}

This linear recurrence is easily implemented; it is what was used to compute the formulas for some small sets $S$ in Table \ref{tab:lmformulas}.

\begin{table}[htbp]
\[
\begin{array}{c| c | c|c}
 S & p_S(n) & p_S(\max S) =p_S(n)/2^{n-\max S} & p_S(7)\\
 \hline
 \hline
 \emptyset & 2^{n-1} & - & 64\\
\{ 3\} & 2^{n-2} & 2 & 32\\
\{ 4\} & 3\cdot 2^{n-2}& 3\cdot 2^2 & 96\\
\{ 5\} & 7\cdot 2^{n-2} & 7\cdot 2^3 & 224\\
\{ 6\} & 5\cdot 3\cdot 2^{n-2} & 5\cdot 3\cdot 2^4 & 480 \\
\{ 7\} & 31\cdot 2^{n-2} & 31\cdot 2^5 & 992\\
\{ 3,5\} & 2^{n-3} & 2^2 & 16 \\
\{ 4,5\} & 3\cdot 2^{n-3} & 3\cdot 2^2 & 48\\
\{ 3,6\} & 3\cdot 2^{n-3} & 3\cdot 2^3 & 48\\
\{ 4,6\} & 3^2\cdot 2^{n-3} & 3^2\cdot 2^3 & 144\\
\{ 5,6\} & 3^2\cdot 2^{n-2} & 3^2 \cdot 2^4 & 288\\
\{ 3,7\} & 7\cdot 2^{n-3} & 7 \cdot 2^4 & 112\\
\{ 4,7\} & 7\cdot 3\cdot 2^{n-3} & 7\cdot 3\cdot 2^4 & 336\\
\{ 5,7\} & 43\cdot 2^{n-3} & 43\cdot 2^4 & 688\\
\{ 6,7\} & 5^2\cdot 3\cdot2^{n-3} & 5^2\cdot 3 \cdot 2^4 & 1200\\
\{ 3,5,7\} & 2^{n-4} & 2^3 & 8\\
\{ 3,6,7\} & 3\cdot 2^{n-4} & 3\cdot2^3 & 24\\
\{ 4,5,7\} & 3\cdot 2^{n-4} & 3\cdot 2^3 & 24\\
\{ 4,6,7\} & 3^2\cdot 2^{n-4} & 3^2\cdot 2^3 & 72\\
\{ 5,6,7\} & 3^2 \cdot 2^{n-3} & 3^2 \cdot 2^4 & 144
\end{array}
\]
\caption{Some formulas for admissible pinnacle sets with $\max S \leq 7$. The formulas are only valid when $n\geq \max S$. The rightmost column has each of these evaluated at $n=7$ for the sake of comparison.}\label{tab:lmformulas}
\end{table}

\subsection{Some formulas and bounds}\label{subsec: some formulas}

The previous discussion leads to a nice result on the bounds of $p_S(n)$. For instance, in Table \ref{tab:lmformulas} it seems that for fixed $d=|S|$, the pinnacle set that maximizes $p_S(n)$ is the one that consists of the largest $d$ elements in $[n]$ (that is, $S = \{n-d+1, n-d+2, \ldots, n\}$). 
In fact, this is true, and we have an explicit formula for $p_S(n)$ in this case.

We begin with the enumeration.

\begin{prop}[Enumerating permutations with maximal pinnacles]\label{prop:stirling formula}
Let $d$ and $n$ be any positive integers such that $2d < n$.
Then the number of permutations in $S_n$ with pinnacle set $[n-d+1,n]=\{n-d+1,n-d+2,\ldots,n\}$ is
\[p_{[n-d+1,n]}(n) = d!\cdot (d+1)! \cdot 2^{n-2d-1} \cdot S(n-d,d+1)\] where $S(\cdot,\cdot)$ denotes the Stirling number of the second kind.
\end{prop}

\begin{proof}
Let $w$ be a permutation in $S_n$ with pinnacle set $S=\{n-d+1,\ldots,n\}$.
Then $w$ has exactly $d$ peaks. Since the $d$ elements in $S$ are each greater than any of the elements of $[n]\setminus S$, these pinnacles are independent of whatever non-pinnacle values are above them.
To construct such a $w$, we can start by ordering the elements of $S$ as pinnacles in $w$, in $d!$ ways.

There are $n-d$ remaining elements to place in the $d+1$ regions around these $d$ peaks.
The Stirling number $S(n-d,d+1)$ counts the number of set partitions of $[n]\setminus S$ into $d+1$ nonempty subsets. Given such a set partition, there are $(d+1)!$ ways to order the subsets, i.e., to choose which subset goes in which region around the peaks.

Finally, it must be the case that the elements in the regions between peaks are arranged in such a way that there are no new peaks. We know by Lemma~\ref{lem:null/reduction} that if there are $k$ elements in a given region, then there are $p_{\emptyset}(k) = 2^{k-1}$ permutations of these elements that have no peaks. 

Let $k_1,\ldots,k_{d+1}$ be the sizes of the subsets in each region between peaks. The product across all regions is
\[
 \prod_{i=1}^{d+1} 2^{k_i-1} = 2^{\sum k_i -(d +1)} = 2^{n-2d-1},
\]
since $\sum k_i = n-d$ is the total number of elements that are not peaks.
\end{proof}

Next we will show that $p_{[n+1-d,n]}(n)$ is maximal among all admissible pinnacle sets having $d$ elements. We preface that work with a lemma that will aid an inductive argument.

\begin{lemma}[Lifting property]\label{lem:lift}
Suppose that $S$ and $T$ are admissible pinnacle sets with $|S|=|T|$, neither of which contains $n$. Then
\[
\mbox{ if } p_S(n-1) \leq p_T(n-1), \mbox{ then } p_{S \cup \{n\}}(n) \leq p_{T\cup \{n\}}(n). 
\]
\end{lemma}

\begin{proof}
Suppose that $|S|=|T|=d$. By the argument that precedes Proposition \ref{prp:linear}, consider any permutation $u \in S_{n-1}$ having $d$ peaks. We have $n-2d$ gaps into which we can insert $n$ to get a permutation with $d+1$ peaks, such that $n$ is a pinnacle. Thus, because $S$ and $T$ each have $d$ elements,
\[
p_{S \cup \{n\}}(n) = (n-2d)p_S(n-1) \text{ \ \ and \ \ } p_{T\cup \{n\}}(n) = (n-2d)p_T(n-1),
\]
yielding the desired implication.
\end{proof}

The following result establishes the upper bound in Theorem \ref{thm:bounds}. More precisely, the following result will describe the pinnacle sets that are achieved most frequently by permutations in $S_n$, and Proposition~\ref{prop:stirling formula} gave the corresponding enumeration.

\begin{prop}[Upper bounds]\label{prop:upper}
Let $d$ and $n$ be any positive integers such that $2d < n$. Then for any admissible pinnacle set $S \subseteq [n]$ with $|S|=d$, we have
\[
 p_S(n) \leq p_{[n+1-d,n]}(n).
\]
\end{prop}

\begin{proof}
We proceed by induction on $n$ and $d$.

The enumeration of admissible pinnacle sets with one element, given in Equation~\eqref{eq:onepeak}, shows that this bound holds in the case when $d=1$ and $n> 2$.

Now suppose that the inequality holds for admissible pinnacle sets that are subsets of $[n-1]$ and that have cardinality less than $(n-1)/2$.

Let $S \subseteq [n]$ be an admissible pinnacle set of cardinality $d$. If $n \in S$, then write $S = S'\cup \{n\}$ for $S' \subseteq [n-1]$. Suppose that $d < n/2$. Then since $|S'|=d-1 < n/2-1 < (n-1)/2$, we can claim, by the inductive hypothesis, that
\[
 p_{S'}(n-1) \leq p_{[n+1-d,n-1]}(n-1).
\]
Now by the lifting property in Lemma \ref{lem:lift}, we have
\[
 p_S(n) \leq p_{[n+1-d,n]}(n),
\]
as desired.

If $n\notin S$ and $d < (n-1)/2$, then Lemma~\ref{lem:null/reduction} yields $p_S(n) = 2p_S(n-1)$. Hence, the induction hypothesis shows that
\[
 p_S(n) = 2p_S(n-1)\leq 2p_{[n-d,n-1]}(n-1).
\]
Further, using our explicit formula from Proposition \ref{prop:stirling formula}, we have
\begin{align*}
2p_{[n-d,n-1]}(n-1) &= 2\left(d!(d+1)!2^{n-2d-2}S(n-1-d,d+1)\right)\\
 &=d!(d+1)!2^{n-2d-1}S(n-1-d,d+1)\\
 &<d!(d+1)!2^{n-2d-1}S(n-d,d+1) = p_{[n+1-d,n]}(n).
\end{align*}

If $n$ is even, then we are done. But if $n$ is odd, then we must also consider the case where $d=(n-1)/2$.

Suppose that $|S| = d = (n-1)/2$. Further suppose that $u \in S_n$ is a permutation of $n=2d+1$ elements having $d$ peaks. Then 
\[
 u(1) < u(2) > u(3) < \cdots >u(2d-1) < u(2d) > u(2d+1).
\]
With this structure, the letter $n$ must be a pinnacle of $u$. Hence if $n\notin S$ and $|S|=(n-1)/2$, then $S$ is not an admissible pinnacle. Thus, $p_S(n) = 0$ and the result follows trivially.
\end{proof}

Next we will prove that for admissible pinnacle sets with $d$ elements, the one that minimizes $p_S(n)$ (that is, the one achieved least often by permutations in $S_n$) is the admissible pinnacle set whose elements are as small as possible. This is the set $\{3,5,\ldots,2d+1\}$. Let us denote this minimizing set
\[M_d := \{ 2k+1 : k=1,\ldots,d\}.\]
We have the following enumerative result.

\begin{prop}[Enumerating permutations with minimal pinnacles]\label{prp:lowcount}
Let $d$ and $n$ be any positive integers such that $2d < n$. Then the number of permutations in $S_n$ with pinnacle set $M_d$ is
\[
 p_{M_d}(n) = 2^{n-d-1}.
\]
\end{prop}

\begin{proof}
The formula is a direct application of the linear recurrence in Equation~\eqref{eq:rec2}, noting that the second summand ranges over an empty set. Hence the sets $M_d$ yield this recurrence:
	\[
	 p_{M_d}(2d+1) = (2d+1-2d)p_{M_{d-1}}(2d) = p_{M_{d-1}}(2d) = 2p_{M_{d-1}}(2d-1),
	\]
with base case $p_{\{3\}}(3) = 2$. Hence $p_{M_d}(2d+1) = 2^d$, and for $n\geq 2d+1$, we use Lemma~\ref{lem:null/reduction} to obtain
\[
 p_{M_d}(n) = 2^{n-d-1},
\]
as desired.
\end{proof}

Alternatively, one could prove Proposition~\ref{prp:lowcount} by explicitly constructing such a permutation. For, if $w \in S_{2d+1}$ has $\Pin(w) = \{3,5,\ldots,2d+1\}$, then $w$ has a simple structure: either $w = (2d)(2d+1)w'$ or $w=w'(2d+1)(2d)$, where $w'$ has $\Pin(w') = M_{d-1}$. This choice of two options at each of $d$ steps gives rise to $2^d$ such permutations. 

\begin{example}
The permutation $w'=13254$ has $\Pin(w') = \{3,5\}$, and there are only two ways to insert $6$ and $7$ in $w'$ to form a permutation $w \in S_7$ with $\Pin(w) = \{3,5,7\}$: either $w = 6713254$ or $w = 1325476$.
\end{example}

If $w\in S_n$ has $\Pin(w) = M_d$, with $n > 2d+1$, then any numbers larger than $2d+1$ have the choice of going on the far left or far right of the permutation, as in the discussion prior to Lemma~\ref{lem:null/reduction}. That is, $w = u w' v$, where $w'\in S_{2d+1}$ has $\Pin(w') = M_d$, the elements of $u$ are decreasing, and the elements of $v$ are increasing.

We will keep this structure in mind for the proof of the following result, which establishes the lower bound in Theorem \ref{thm:bounds}.

\begin{prop}[Lower bounds]\label{prp:lower}
	Let $d$ and $n$ be any positive integers such that $2d < n$. Then for any admissible pinnacle set $S\subseteq[n]$ with $|S|=d$, we have
\[
		p_S(n) \geq  p_{M_d}(n) = 2^{n-d-1}.
\]
\end{prop}

\begin{proof}
	Fix $d$ and $n>2d$. Let $S\subseteq [n]$ be an admissible pinnacle set with $|S|=d$. 
	Let $A$ denote the set of permutations in $S_n$ with pinnacle set $M_d$, and let $B$ denote the set of permutations in $S_n$ with pinnacle set $S$. We will construct an injection from $A$ to $B$ as follows.
	
	Let $w \in A$. Then $w = uw'v$, a concatenation of strings, where $w' \in S_{2d+1}$ has $\Pin(w') = M_d$, $u$ is a list of decreasing elements, and $v$ is a list of increasing elements.
	Now order the elements of set $S = \{ s_1 < s_2 < \cdots < s_d\}$, and recall that $s_k \geq 2k+1$ for each $k = 1,\ldots, d$. We will define the permutation $\widehat{u}\ww'\widehat{v}=\ww \in B$ as follows. 
	
	First, replace the pinnacles of $w'$ with the elements of $S$, in the same relative order. That is, if $w'(j) = 2k+1$, then $\ww'(j) = s_k$. 
	Second, replace the remaining elements of $w$, i.e., its non-pinnacles, with the elements of $[n] \setminus S$ in the same relative order as the non-pinnacles of $w$.
	From this construction we have formed $\widehat{u}$ and $\widehat{v}$ by placing the elements of $[n]\setminus \{\ww'(i)\} =\{b_1 < \cdots < b_{n-2d-1}\}$ in the same positions and same relative order as the elements of $[n]\setminus \{w'(i)\} = \{a_1 < \cdots < a_{n-2d-1}\}$ had been in $w$. That is, each $a_i$ is replaced by $b_i$.
	
	For example, consider $M_2 = \{3,5\}$. One permutation in $S_9$ with pinnacle set $\{3,5\}$ is $w = 813254679$. Here we have $u=8$, $w'= 13254$, and $v=679$. If $S = \{4,6\}$, then we first replace $3$ by $4$ and $5$ by $6$.
	Then, we place the letters of $[9]\setminus\{4,6\}$ in $\ww$ in the same relative order as the letters of $[9] \setminus \{3,5\}$ in $w$:
	\[
	 \begin{array}{cccccccccc}
	 w= & 8 & 1 & 3 & 2 & 5 & 4 & 6 & 7 & 9 \\
	 & & & \downarrow & & \downarrow & \\ 
	 & \cdot & \cdot & \mathbf{4} & \cdot & \mathbf{6} & \cdot & \cdot & \cdot & \cdot \\
	 & \downarrow & \downarrow & & \downarrow & & \downarrow &  \downarrow & \downarrow & \downarrow \\
	 \ww= & \mathbf{8} & \mathbf{1} & 4 & \mathbf{2} & 6 & \mathbf{3} & \mathbf{5} & \mathbf{7} & \mathbf{9}
	 \end{array}
	\]
yielding $\ww = \widehat{u}\ww'\widehat{v} = 814263579$.
	
	The construction of this permutation $\ww$ guarantees that $\ww \in B$. 
	 The only other pinnacles that $\ww$ could have would be in $\widehat{u}$, in $\widehat{v}$, or at either end of $\ww'$. However, the strings $\widehat{u}$ and $\widehat{v}$ are monotonic, so they contain no peaks, while the left end of $\ww'$ is an ascent and the right end of $\ww'$ is preceded by a descent, so these cannot be peaks either. Therefore $\ww \in B$.
	
	We claim that the map $w = uw'v \mapsto \widehat{u}\ww'\widehat{v}=\ww$ is an injection from $A$ to $B$. 
	Indeed, if $w$ and $x$ are two different permutations in $A$, then their peak sets are either distinct or identical.
	If their peak sets are distinct, then the peak sets of their images in $B$ are also distinct, and so $w$ and $x$ must map to distinct elements in $B$.
	If their peak sets are the same, then the remaining elements are in distinct relative orders, and so their images in $B$ also have their elements of $[n] \setminus S$ in distinct relative orders, hence the images of $w$ and $x$ are distinct.
	Therefore, the map is indeed injective.
\end{proof}

The results in this section allow us to find admissible pinnacle sets $S$ that maximize and minimize $p_S(n)$, for fixed $n$. For the lower bound, we have
\[
 \min\{p_S(n): \mbox{admissible } S\subseteq [n] \} = \min\{ 2^{n-d-1} : d < n/2\} = 2^{\lfloor n/2 \rfloor}.
\]
For the upper bound, we have something a little less satisfying:
\[
 \max\{p_S(n): \mbox{admissible } S\subseteq [n] \} = \max\{ d!(d+1)!2^{n-2d-1}S(n-d,d+1) : d < n/2\}.
\]
This introduces an interesting statistic.

\begin{defn}\label{defn:maximizing with stirling}
For fixed $n$, let $d(n) = d < n/2$ be the value maximizing the expression $d!(d+1)!2^{n-2d-1}S(n-d,d+1)$.
\end{defn}

We can compute $d(n)$ for small values of $n$, and some of this data appears in Table \ref{tab:maxdata}.

\begin{table}[htbp]
\[
\begin{array}{c|| ccccccccccccccccccccc}
\raisebox{.1in}[.2in][.1in]{}n  &  4 & 5 & 6 & 7 & 8 & 9 & 10 & 11 & 12 & 13 & 14 & 15 & 16 & 17 & 18 & 19 & 20 & 21 & 22\\
\hline
\raisebox{.1in}[.2in][.1in]{}d(n) & 1& 1& 1 & 2& 2& 2& 3 & 3 & 3& 4& 4& 4& 4& 5& 5& 5& 6& 6& 6  \\
\end{array}
\]
\[
\begin{array}{c ccccccccccccccccccccc}
\raisebox{.1in}[.2in][.1in]{}\cdots  &  5000 & 5001 & 5002 & 5003 & 5004 & 5005 & 5006 & 5007 & 5008 & 5009 & 5010 \\
\hline
\raisebox{.1in}[.2in][.1in]{}\cdots & 1598 & 1598 & 1599 & 1599& 1599& 1600& 1600 & 1600 & 1600& 1601& 1601  \\
\end{array}
\]
\caption{The value $d(n)$ that maximizes $d!(d+1)!2^{n-2d-1}S(n-d,d+1)$, producing the maximum value for $p_S(n)$ across all admissible pinnacle sets $S \subseteq [n]$.}\label{tab:maxdata}
\end{table}

Initially, $d(n)$ appears to be a step function that increases by one as $n$ increases by three. But $d(16) = 4$ shows that this pattern is false in general. In Figure~\ref{fig:maxdata}, we plot the function $d(n)$ for $n\leq 200$. A first look at this picture suggests that the step function cycles through seven plateaus of width three and an eighth plateau of width four, but this pattern also does not persist. For example, $d(n)=12$ for the four consecutive values from $n=38$ to $n=41$ and $d(n)=20$ for the four consecutive values from $n=63$ to $n=66$. But the next plateau of four is only seven steps away: $d(n) = 27$ from $n=85$ to $n=88$.

\begin{figure}[htbp]
\includegraphics[scale=.7]{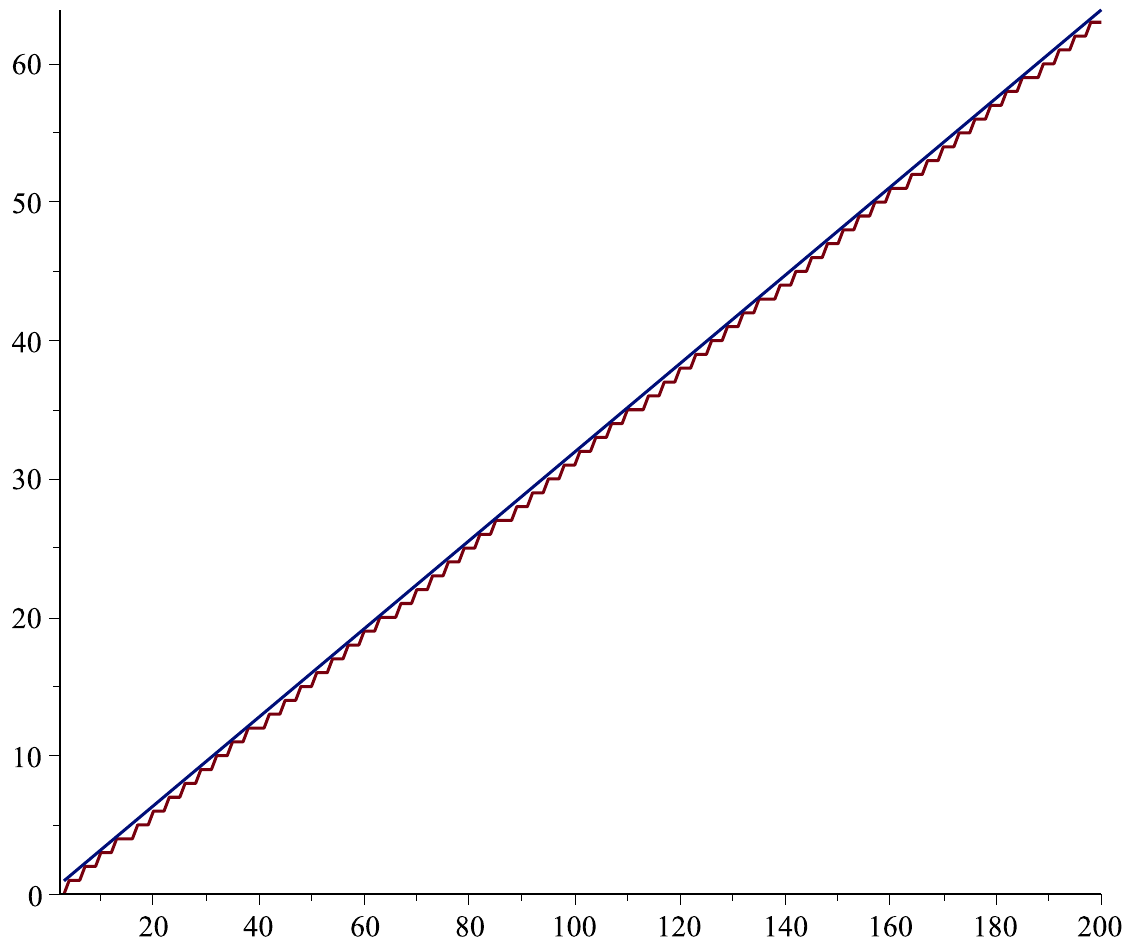}
\caption{The step function $d(n)$ for $n\leq 200$, plotted alongside the line $d=n/3.13$. Some plateaus have width four; all others have width three.}\label{fig:maxdata}
\end{figure}

In Table \ref{tab:dn} we list the values of $n$ and $d(n)$ for which there are four consecutive values $n$ with the same $d(n)$, i.e., for which $\{d(n), d(n+1), d(n+2), d(n+3)\}$ is a set of size $1$. All other values of $d(n)$ with $n\leq 200$ come in runs of three. The fact that the plateaus of size four are not quite periodic is puzzling.

\begin{table}[htbp]
\[
\begin{array}{c | ccccccccc}
\raisebox{.1in}[.2in][.1in]{} n & 13 & 38 & 63 & 85 & 110 & 135 & 160 & 185 \\
 \hline
\raisebox{.1in}[.2in][.1in]{} d(n) & 4 & 12 & 20 & 27 & 35 & 43 & 51 & 59
\end{array}
\]
\caption{The values of $n\leq 200$ and corresponding $d(n)$ that mark the beginnings of four consecutive equal values: $d(n)=d(n+1)=d(n+2)=d(n+3)$.}\label{tab:dn}
\end{table}

While it seems that $d(n)$ is approximately $n/3$, an exact formula for $d(n)$ (and hence the maximal value for $p_S(n)$) is so far elusive. By checking values up to around $n=5000$, we see $n/d(n) \approx 3.13$. See Figure \ref{fig:ndn}, in which we have plotted $n/d(n)$ for $n\leq 5000$. Note that the ratio $n/d(n)$ is not monotonic. In particular, the minimum in this domain is $4786/1530 \approx 3.1281$ and the maximum is $4004/1279 \approx 3.1306$.

\begin{figure}[htbp]
\includegraphics[scale=.7]{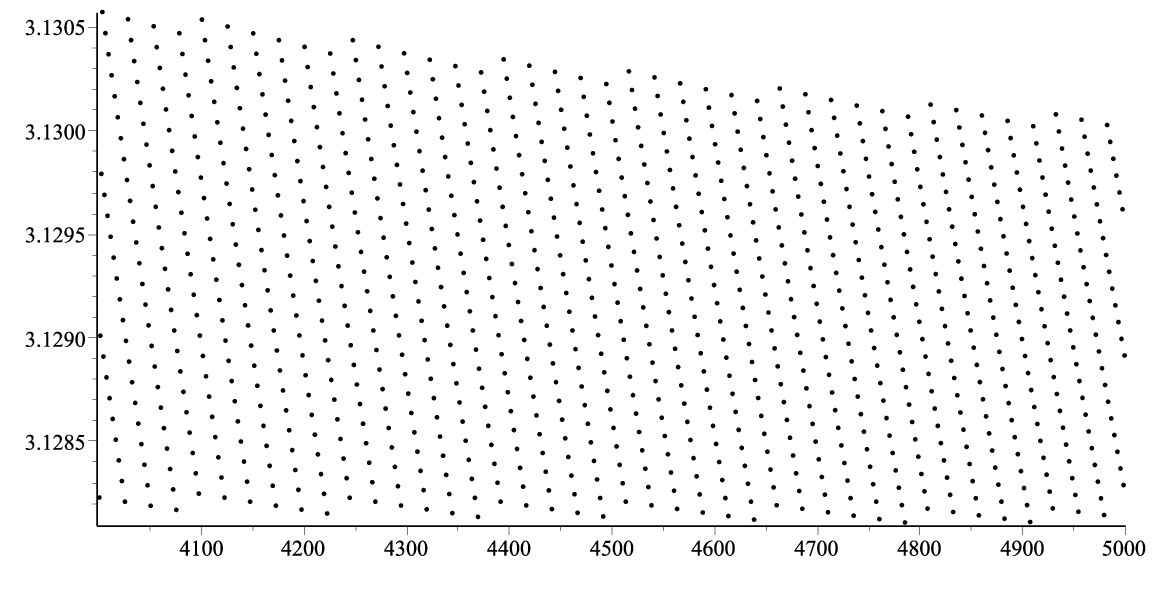}
\caption{Values for $n/d(n)$ with $4000\leq n \leq 5000$. For these values of $n$, $3.1281< n/d(n) < 3.1306$.}\label{fig:ndn}
\end{figure}

\section{Further questions}\label{sec:conclude}

The results proven in this paper are a small sample of the directions in which the study of pinnacle sets may be taken. The first question we pose here is the same one with which we closed the last section.

\begin{question}
	Is there a simple formula for $d(n)$, as introduced in Definition~\ref{defn:maximizing with stirling}? What is $d(n)$ asymptotically?
\end{question}

While there are good asymptotics for Stirling numbers in certain regimes (e.g., Bleick and Wang \cite{BleickWang} give estimates for $S(n,k)$ when either $k$ or $n-k$ grows slowly compared to $n$) these estimates are not obviously helpful here since we apparently need to understand $S(n-d,d+1)$ near $n=3.13d$. That is, we want good asymptotics for something like $S(cd,d+1)$, where $c \approx 2.13$.

Another question seeks to explore nontrivial ways in which permutations with the same pinnacle set are related.

\begin{question}
	For a given $S$, is there a class of operations (e.g., valley hopping as in \cite{Branden}) that one may apply to any $w \in S_n$ with $\Pin(w) = S$ to obtain any other permutation $w' \in S_n$ with $\Pin(w') = S$, and no other permutations?
\end{question}

Among the admissible pinnacle sets $S\subseteq [n]$ of a fixed size, we know which sets $S$ minimize $p_S(n)$ and which maximize $p_S(n)$. However, it seems trickier to compare two randomly selected sets. For example, with $n=7$, here are the $2$-element admissible subsets of $[7]$ ordered according to $p_S(7)$:
\begin{align*}
 p_{\{3,5\}}(7) < p_{\{4,5\}}(7) = p_{\{3,6\}}&(7) < p_{\{3,7\}}(7) \\
 &< p_{\{4,6\}}(7) < p_{\{5,6\}}(7) < p_{\{4,7\}}(7) < p_{\{5,7\}}(7) < p_{\{6,7\}}(7).
\end{align*}
The linear ordering here seems difficult to explain, but a partial ordering on sets that is compatible with comparison might be more feasible. For example the coordinate-wise dominance order shown below is compatible with the ordering on $p_S(n)$.
\[
 \begin{tikzpicture}
  \draw (0,0) node[inner sep=2] (35) {35};
  \draw (1,1) node[inner sep=2] (36) {36};
  \draw (-1,1) node[inner sep=2] (45) {45};
  \draw (0,2) node[inner sep=2] (46) {46};
  \draw (2,2) node[inner sep=2] (37) {37};
  \draw (1,3) node[inner sep=2] (47) {47};
  \draw (0,4) node[inner sep=2] (57) {57};
  \draw (-1,3) node[inner sep=2] (56) {56};
  \draw (-1,5) node[inner sep=2] (67) {67};
  \draw (35)--(36)--(37)--(47)--(57)--(67);
  \draw (35)--(45)--(46)--(47);
  \draw (36)--(46)--(56)--(57);
 \end{tikzpicture}
\]

\begin{question}
 Is there a nontrivial partial order on admissible pinnacle sets such that if $S\leq T$ in the partial order, then $p_S(n) \leq p_T(n)$?
\end{question}
 
In Section~\ref{sec:rec} we established certain recursive formulas for $p_S(n)$, but we only had explicit formulas in a few special cases, such as those used to prove our upper and lower bounds. Perhaps it is possible to do better.

\begin{question}
	For general $n$ and $S$, is there a closed-form, non-recursive formula for $p_S(n)$? 
\end{question}

As a step in this direction, notice that combining the formulas \eqref{eq:onepeak} and \eqref{eq:2peak} from Proposition \ref{prp:12peak} yields the following:
\begin{align*}
p_{\emptyset}(n) &= 2^{n-1},\\
p_{\emptyset}(n) + 2p_{\{l\}}(n) &= 2^{n+l-3},\\
p_{\emptyset}(n) + 2p_{\{l\}}(n) + 2p_{\{m\}}(n) + 4p_{\{l,m\}}(n) &= 2^{n+m-l-3}(3^{l-1}+1).
\end{align*}
It is not completely clear what the pattern might be here, but perhaps for an admissible pinnacle set $S$, the quantity $q_S(n)$ defined as follows,
\[
 q_S(n) = \sum_{I\subseteq S} 2^{|I|}p_I(n),
\]
might be well-behaved. If so, this would give an inclusion-exclusion formula for $p_S(n)$.

\begin{question}
	For general $n$ and $S$, is there a closed-form, non-recursive formula for $q_S(n)$? 
\end{question}

\appendix
\section{Descent topsets}\label{sec:topsets}

In this section we describe some of the enumerative properties of descent topsets as mentioned in Remark \ref{rem:topsets}. The reader interested only in pinnacle sets can safely skip to the next section.

Let $d_I(n)$ denote the number of permutations in $S_n$ for which $\Dtop(w) = I$, while $e_I(n)$ is the number of permutations with $\Dtop(w)\subseteq I$. That is,
\[
 d_I(n) = | \{ w \in S_n : \Dtop(w) = I\}|,\]
 and 
\[ 
e_I(n) = |\{ w \in S_n : \Dtop(w) \subseteq I\}| =\sum_{J\subseteq I} d_J(n).
\]
Notice that these functions are piecewise constant. That is, if $n < \max I$ then $d_I(n) = 0$, while if $n \geq m= \max I$, then $d_I(n) = d_I(m)$. This follows from the observation that for $w \in S_n$ with $\Dtop(w) = I$, all the letters larger than $m$ must appear in increasing order at the far right of $w$, else they create a new descent top. 

Thus for sufficiently large $n$ we can ignore the dependence on $n$ and it follows from inclusion-exclusion that
\[
 d_I = \sum_{J \subseteq I} (-1)^{|I-J|} e_J.
\]
We now give a formula for the $e_J$. 

To describe the formula for $e_J$, we introduce some notation. For a set $S = \{s_1,s_2,\ldots,s_k\}$, define the composition $\alpha(S)= (s_1-1, s_2-s_1, s_3-s_2,\ldots,s_k-s_{k-1})$. Then for $\alpha=(\alpha_1,\ldots,\alpha_k)$ denote by $\alpha!$ the product
\[
 (k+1)^{\alpha_1}k^{\alpha_2}\cdots 3^{\alpha_{k-1}}2^{\alpha_k}.
\]
For example $(2,3,2)! = 4^23^32^2=1728$. If $S=\emptyset$, then $\alpha(\emptyset) = ()$ and $()! = 1$. 

The enumerative result is the following.

\begin{thm}[See Theorem 4.1 of \cite{NovelliThibonWilliams}, Lemma 6.1 of \cite{EhrenborgSteingrimsson}]\label{thm:eJ}
Suppose $n\geq \max J$. Then
\[
 e_J(n)=e_J = \alpha(J)! \quad \mbox{ and } \quad d_I = \sum_{J \subseteq I} (-1)^{|I-J|}\alpha(J)!.
\]
\end{thm}

As an example, if $I = \{3,6,8\}$, we can use Theorem \ref{thm:eJ} to compute
 \begin{align*}
  d_I &= \alpha(\{3,6,8\})! - \alpha(\{3,6\})! - \alpha(\{3,8\})! - \alpha(\{6,8\})! \\
  & + \alpha(\{3\})! + \alpha(\{6\})! + \alpha(\{8\})! - \alpha(\emptyset)!,\\
   &= (2,3,2)! - (2,3)! - (2,5)! - (5,2)! + (2)! + (5)! + (7)! - ()!,\\
   &= 4^2 3^3 2^2 - 3^2 2^3 - 3^22^5 - 3^52^2 + 2^2 + 2^5 + 2^7 - 1,\\
   &= 559.
 \end{align*}

\begin{proof}
The theorem can be proved recursively by thinking of building permutations ``from the top down." That is, suppose we are building a permutation with descent topset contained in $J= \{j_1 < \cdots <j_k\}$. Let $w'$ denote a permutation of $\{j_1,j_1+1,\ldots,n\}$ with descent topset contained in $J'=J-\{j_1\}$. By induction, the number of such $w'$ is $e_{J'}$, and we will argue that $e_J = (k+1)^{j_1-1}e_{J'}$, from which the result follows.

Given $w'$ we will form a permutation $w \in S_n$ with $\Dtop(w) \subseteq J$ by inserting the numbers $\{1,2,\ldots,j_1-1\}$ in the gaps between the potential descent tops indicated by $J$. Since $|J|=k$, there are $k+1$ gaps in which to place these elements, and in each gap the elements must be written in increasing order. Thus there are $(k+1)^{j_1-1}$ ways to form $w$ given $w'$. Hence $e_J = (k+1)^{j_1-1}e_{J'}$ as desired.
\end{proof}

It is interesting to compare this approach to descent top enumeration with the classical result of MacMahon for ordinary descent sets (see Proposition 1.4.1 of \cite{StanleyEC1}), which says the number of permutations in $S_n$ with $\Des(w)\subseteq I = \{i_1 < i_2 < \cdots <i_k\}$ is the multinomial coefficient
\[
 \binom{ n}{i_1, i_2-i_1, i_3-i_2,\ldots,n-i_k}.
\]
(Arrange numbers into increasing runs of lengths $i_1, i_2-i_1, i_3-i_2$, and so on.) By inclusion-exclusion, the number of permutations in $S_n$ with $\Des(w) = I$ is
\[
 \sum_{1\leq r_1 < \cdots < r_j \leq k} (-1)^{k-j}\binom{ n}{i_{r_1}, i_{r_2}-i_{r_1},\ldots,n-i_{r_j}}.
\]

\bibliographystyle{plain}
\bibliography{pinnacle}

\begin{thebibliography}{10}

\bibitem{AguiarNymanOrellana}
Marcelo Aguiar, Kathryn Nyman, and Rosa Orellana.
\newblock New results on the peak algebra.
\newblock {\em J. Algebraic Combin.}, 23(2):149--188, 2006.

\bibitem{BilleraHsiaoVW}
Louis~J. Billera, Samuel~K. Hsiao, and Stephanie van Willigenburg.
\newblock Peak quasisymmetric functions and {E}ulerian enumeration.
\newblock {\em Adv. Math.}, 176(2):248--276, 2003.

\bibitem{BilleyBurdzySagan}
Sara Billey, Krzysztof Burdzy, and Bruce~E. Sagan.
\newblock Permutations with given peak set.
\newblock {\em J. Integer Seq.}, 16(6):Article 13.6.1, 18, 2013.

\bibitem{BilleyFahrbachTalmage}
Sara Billey, Matthew Fahrbach, and Alan Talmage.
\newblock Coefficients and roots of peak polynomials.
\newblock {\em Exp. Math.}, 25(2):165--175, 2016.

\bibitem{BleickWang}
W.~E. Bleick and Peter C.~C. Wang.
\newblock Asymptotics of {S}tirling numbers of the second kind.
\newblock {\em Proc. Amer. Math. Soc.}, 42:575--580, 1974.

\bibitem{Branden}
Petter Br{\"a}nd{\'e}n.
\newblock Actions on permutations and unimodality of descent polynomials.
\newblock {\em European J. Combin.}, 29(2):514--531, 2008.

\bibitem{DiazHarrisInskoOmar}
Alexander Diaz-Lopez, Pamela~E. Harris, Erik Insko, and Mohamed Omar.
\newblock A proof of the peak polynomial positivity conjecture.
\newblock {\em J. Combin. Theory Ser. A}, 149:21--29, 2017.

\bibitem{DiazHarrisInskoOmarSagan}
Alexander Diaz-Lopez, Pamela~E. Harris, Erik Insko, Mohamed Omar, and Bruce~E.
  Sagan.
\newblock Descent polynomials, 2017.

\bibitem{EhrenborgSteingrimsson}
Richard Ehrenborg and Einar Steingr{\'\i}msson.
\newblock The excedance set of a permutation.
\newblock {\em Adv. in Appl. Math.}, 24(3):284--299, 2000.

\bibitem{EhrenborgSteingrimssonAlternating}
Richard Ehrenborg and Einar Steingr{\'\i}msson.
\newblock Yet another triangle for the {G}enocchi numbers.
\newblock {\em European J. Combin.}, 21(5):593--600, 2000.

\bibitem{FoataZ}
Dominique Foata and Doron Zeilberger.
\newblock Denert's permutation statistic is indeed {E}uler-{M}ahonian.
\newblock {\em Stud. Appl. Math.}, 83(1):31--59, 1990.

\bibitem{GarsiaReutenauer}
A.~M. Garsia and C.~Reutenauer.
\newblock A decomposition of {S}olomon's descent algebra.
\newblock {\em Adv. Math.}, 77(2):189--262, 1989.

\bibitem{HNTT}
Florent Hivert, Jean-Christophe Novelli, Lenny Tevlin, and Jean-Yves Thibon.
\newblock Permutation statistics related to a class of noncommutative symmetric
  functions and generalizations of the {G}enocchi numbers.
\newblock {\em Selecta Math. (N.S.)}, 15(1):105--119, 2009.

\bibitem{Kasraoui}
Anisse Kasraoui.
\newblock The most frequent peak set of a random permutation, 2012.

\bibitem{KitaevMansourRemmel}
Sergey Kitaev, Toufik Mansour, and Jeff Remmel.
\newblock Counting descents, rises, and levels, with prescribed first element,
  in words.
\newblock {\em Discrete Math. Theor. Comput. Sci.}, 10(3):1--22, 2008.

\bibitem{NovelliThibonWilliams}
J.-C. Novelli, J.-Y. Thibon, and L.~K. Williams.
\newblock Combinatorial {H}opf algebras, noncommutative {H}all-{L}ittlewood
  functions, and permutation tableaux.
\newblock {\em Adv. Math.}, 224(4):1311--1348, 2010.

\bibitem{Nyman}
Kathryn~L. Nyman.
\newblock The peak algebra of the symmetric group.
\newblock {\em J. Algebraic Combin.}, 17(3):309--322, 2003.

\bibitem{Schocker}
Manfred Schocker.
\newblock The peak algebra of the symmetric group revisited.
\newblock {\em Adv. Math.}, 192(2):259--309, 2005.

\bibitem{Solomon}
Louis Solomon.
\newblock A {M}ackey formula in the group ring of a {C}oxeter group.
\newblock {\em J. Algebra}, 41(2):255--264, 1976.

\bibitem{StanleyEC1}
Richard~P. Stanley.
\newblock {\em Enumerative combinatorics. {V}olume 1}, volume~49 of {\em
  Cambridge Studies in Advanced Mathematics}.
\newblock Cambridge University Press, Cambridge, second edition, 2012.

\bibitem{SteinWilliams}
Einar Steingr{\'{\i}}msson and Lauren~K. Williams.
\newblock Permutation tableaux and permutation patterns.
\newblock {\em J. Combin. Theory Ser. A}, 114(2):211--234, 2007.

\end{thebibliography}

\end{document}